\documentclass[12pt]{article}
\usepackage[utf8]{inputenc}
\usepackage[T1]{fontenc}
\usepackage{amsthm,amsmath}
\usepackage{amssymb}
\usepackage{natbib}
\usepackage{fullpage}
\usepackage{mathtools}
\usepackage{graphicx}
\usepackage{color}
\usepackage{tabularx}
\usepackage{float}
\usepackage{array}
\usepackage{multirow}
\usepackage{bm}
\usepackage{algorithm2e}
\usepackage{algpseudocode}
\usepackage{url}

\newtheorem{theorem}{Theorem}
\newtheorem{proposition}[theorem]{Proposition}%
\newtheorem{remark}{Remark}%

\newtheorem{definition}{Definition}%
\newtheorem{assum}[definition]{Assumption}
\newtheorem{cond}[definition]{Condition}

\newtheorem{lemma}[theorem]{Lemma}

\usepackage{tcolorbox}

\title{Tree-based conditional copula estimation}
\author{Francesco Bonacina$^{1,2}$, Olivier Lopez$^{3}$ and Maud Thomas$^{1}$}
\date{}

\begin{document}

\maketitle

{\small \noindent
$^1$ Sorbonne Universit\'e, CNRS, Laboratoire de Probabilit\'es, Statistique et Mod\'elisation, LPSM, 4 place Jussieu, F-75005 Paris, France,\\
$^2$ Sorbonne Universit\'e, INSERM, Institut Pierre Louis d'Epidémiologie et de Santé Publique, F75012 Paris, France\\
$^3$ CREST Laboratory, CNRS, Groupe des Écoles Nationales d'Économie et Statistique, Ecole Polytechnique, Institut Polytechnique de Paris, 5 avenue Henry Le Chatelier 91120 Palaiseau, France \\
E-mails: francesco.bonacina@sorbonne-universite.fr, olivier.lopez@ensae.fr, maud.thomas@sorbonne-universite.fr}
 
\begin{abstract}
This paper proposes a regression tree procedure to estimate conditional copulas. The associated algorithm determines classes of observations based on covariate values and fits a simple parametric copula model on each class. The association parameter changes from one class to another, allowing for non-linearity in the dependence structure modeling. It also allows the definition of classes of observations on which the so-called "simplifying assumption"  \citep[see][]{derumigny2017tests} holds reasonably well. When considering observations belonging to a given class separately, the association parameter no longer depends on the covariates according to our model. In this paper, we derive asymptotic consistency results for the regression tree procedure and show that the proposed pruning methodology, that is the model selection techniques selecting the appropriate number of classes, is optimal in some sense. Simulations provide finite sample results and an analysis of data of cases of human influenza presents the practical behavior of the procedure.
\end{abstract}

\section{Introduction}

 Since Sklar's seminal result, copula theory has emerged as a practical means of describing the dependence between random variables. Allowing one to distinguish between the marginal behavior of each component of a random vector and the dependence structure (represented by a copula function), Sklar's theorem opens the way to flexible modeling of various forms of dependence \cite[see][]{nelsen}. 
In this paper, we propose a new method to perform conditional copula analysis based on regression trees and to derive consistency results for this procedure.

Various estimation procedures and analyses of copulas have been studied in the statistical literature \citep[e.g. see][]{genest,tsukahara, fermanian1, segers,fermanian2}. In the presence of covariates, conditional copula analysis consists in fitting a copula function to the conditional distribution of a random vector. From an application point of view, \citet{dupuis2006multivariate} have shown their importance in modeling certain natural disasters such as hurricanes, or the dependence between different expense lines in actuarial problems.
\citet{lopez} and \citet{farkas2023semiparametric} have used this type of model for insurance claim management. Another important application, for example in finance, can be found in \citet{jaworski2010copula}. More generally, the study of conditional copulas also appears particularly important in Vine copulas, \citep[see][]{czado2022vine}. \citet{abegaz2012semiparametric,gijbels2012multivariate,gijbels2011conditional} have studied both semi-parametric and non-parametric procedures for performing this analysis. Finally, \citet{fermanian_lopez} have examined the case of high-dimensional covariates and relied on a dimension reduction approach to perform the analysis.

We propose here to use regression trees to perform this conditional copula analysis. Regression trees, along with the {\it Classification And Regression Tree} (CART) algorithm, were originally introduced by \citet{breiman}  algorithm and are now classic tools used for several applications  \citep[e.g. see][]{gocheva2019regression, farkas, farkas2021generalized,loh2014fifty}. Apart from the computational efficiency of the  CART estimation algorithm used to fit the model, an interesting feature of this approach is the ability to construct classes of individuals (based on their characteristics) with similar behavior. In the context of copula analysis, this corresponds to classes of individuals with the same copula (i.e. dependence) structure. This model can be seen as a means to easily generalize the ``simplifying'' assumption considered by many authors \citep[see e.g.][]{derumigny2017tests, kurz2022testing}. According to this hypothesis, only the marginal distributions of each component depend on the covariates, while the dependence structure does not vary with them. In contrast, in our model, the copulas are different for each cluster determined by the regression tree, and thus the simplifying assumption holds separately for each cluster.

The rest of the paper is organized as follows. In Section~\ref{sec1}, we describe the general framework of regression trees and the algorithm used to fit them to data. Section~\ref{sec2} is devoted to proving theoretical results on the consistency of this procedure. Particular attention is paid to the part of model selection, known as the ``pruning step", which consists of selecting an appropriate sub-tree from the maximal tree obtained by iterative partitioning of the data set. In Section~\ref{sec4}, the practical behavior is investigated through a simulation study and real data analysis. The proofs of the theoretical results are gathered in the Appendix. 

\section{Regression trees for conditional copula analysis}
\label{sec1}

\subsection{Model and notations}
\label{sec_notation}

 We consider a set of observations $(\mathbf{Y}_i,\mathbf{X}_i)_{1\leq i \leq n}$ consisting of independent identically distributed copies of the random vector $(\mathbf{Y},\mathbf{X}),$ where $\mathbf{X}\in \mathcal{X}\subset \mathbb{R}^d$ are covariates, and $\mathbf{Y}=(Y^{(1)},\ldots,Y^{(k)})\in \mathbb{R}^k$ is a random vector of response variables $Y^{(j)}$, $j=1,\ldots,k$. The marginal conditional cumulative distribution functions (c.d.f.) of the random vector $\mathbf{Y}$ given $\mathbf X= \mathbf x$ are defined as
\[
F^{(j)}(t^{(j)}|\mathbf{x})=\mathbb{P}\left(Y^{(j)}\leq t^{(j)}|\mathbf{X}=\mathbf{x}\right) , \quad t^{(j)}\in \mathbb{R}, \ j=1,\ldots, k.
\]
From Sklar's Theorem \citep{sklar}, the joint conditional c.d.f. $F(\mathbf{t}|\mathbf{x})=\mathbb{P}(\mathbf{Y}\leq \mathbf t|\mathbf{X}=\mathbf{x})$ can be expressed as, for all $ \mathbf t=(t^{(1)},\ldots, t^{(k)}) \in \mathbb R^k$,
\begin{equation}
\label{eq_sklar}
F(\mathbf t|\mathbf{x})=\mathfrak{C}_{\mathbf{x}}(F^{(1)}(t^{(1)}|\mathbf{x}),\ldots,F^{(k)}(t^{(k)}|\mathbf{x})),
\end{equation}
where, for all $\mathbf{x},$ $\mathfrak{C}_{\mathbf{x}}$ is a copula function, that is a c.d.f. on $[0,1]^k$ with margins uniformly distributed over $[0,1]$. The copula function $\mathfrak{C}_{\mathbf{x}}$ in~\eqref{eq_sklar} is unique if the distribution of $\mathbf{Y}$ is continuous, which is the assumption that we will make throughout the paper. In general, the analyses of the marginal distributions and the dependence structure are therefore made separately. 

In the following, we will consider a semi-parametric assumption as in \citep{abegaz2012semiparametric} or \citep{lopez} by introducing a parametric family of copula functions $\mathcal{C}=\{C_{\bm\theta}:\bm\theta\in \Theta\}$ with $\Theta\subset \mathbb{R}^m$. We denote $c_{\bm \theta}$ the copula density associated with $C_{\bm \theta}$, that is
\[
c_{\bm \theta}(\mathbf{u})=\frac{\partial^k C_{\bm\theta}(\mathbf{u})}{\partial u^{(1)}\ldots \partial u^{(k)}} , \quad \bm u=(u^{(1)}, \ldots, u^{(k)}) \in [0,1]^k.
\]

In the sequel, we assume that, for all $\mathbf x \in \mathbb R^d$, $\mathfrak C_{\mathbf x} \in \mathcal C$, meaning that there exists a unique  $\bm\theta^0(\mathbf x) \in \Theta $ such that 
\begin{equation}
 \mathfrak{C}_{\mathbf{x}}=C_{\bm\theta^0(\mathbf{x})}.
\end{equation}
Our aim is then to retrieve  the function $\bm\theta^0(\mathbf{x})$ from the data $(\mathbf{Y}_i,\mathbf{X}_i)_{1\leq i \leq n}$.

Our estimation strategy is based on regression trees. A tree $\mathbb{T}$ of size $K$ is a partition of $\mathcal X$, that is  $\mathbb{T}=(\mathcal{T}_\ell)_{\ell=1,\ldots,K}$ where $\mathcal{T}_\ell\cap \mathcal{T}_{\ell'}=\emptyset$ for $\ell\neq \ell'$ and $\cup_{\ell=1}^K \mathcal{T}_\ell=\mathcal{X}$. The sets $\mathcal T_{\ell}$, $\ell=1,\ldots, K$ are called leaves and each leaf  $\mathcal{T}_\ell$ is obtained as the intersection of conditions of the type $x_{-,\ell}^{(j)}\leq x^{(j)}\leq x_{+,\ell}^{(j)}$ if $X^{(j)}$ is continuous, and of the type $x^{(j)}\in \mathcal{A}_\ell^{(j)}$ where $\mathcal{A}_\ell^{(j)}$ is a set of potential modalities for a discrete covariate. This particular structure of the partition is associated with a binary tree structure, where the nodes of the tree correspond to conditions on a given covariate and the leaves of the tree to the final classification. The CART algorithm described in Section~\ref{sec_cart} will make this tree structure more obvious.

Given a tree $\mathbb{T}$ with $K$ leaves, we thus consider estimators of $\bm\theta^0(\mathbf{x})$ that are constant on each leaf of $\mathbb{T},$ that is, of the type
$\sum_{\ell=1}^K \theta_\ell \mathbf{1}_{\mathcal{T}_\ell}(\mathbf{x})$, with $\theta_\ell \in \mathbb R^m$. In other words, individuals are divided into K classes, for each of which the dependence structure is
described by a different copula (from the same parametric family, but with a specific parameter $\theta_\ell$). In the ideal case, the target function $\bm\theta^0(\mathbf{x})$ is constant on each leaf of the tree $\mathbb{T},$ meaning that $\bm\theta^0(\mathbf{x})=\theta^0_\ell$ for $\mathbf{x}\in \mathcal{T}_\ell,$ where  
\[
\theta^0_\ell=\arg \max_{\bm\theta \in \Theta}\mathbb{E}\left[\log c_{\bm\theta}(\mathbf{U}_i)\mathbf{1}_{\mathbf{X}_i\in \mathcal{T}_\ell}\right],
\]
where $c_{\bm\theta}$ is the copula density associated with the copula function $C_{\bm\theta}$ and $\mathbf U_i$ is the random variable defined by
\[
\mathbf{U}_i=(F^{(1)}(Y_{i}^{(1)}|\mathbf{X}_i),\ldots, F^{(k)}(Y_{i}^{(k)}|\mathbf{X}_i)),
\]
 which has uniform margins over $[0,1]$ and is jointly distributed according to the c.d.f. $\mathfrak{C}_{\mathbf{X}_i}=C_{\bm\theta^0(\mathbf{X}_i)}.$ 
 
 However, in practice, a misspecification bias is expected, since the target function $\bm\theta^0(\mathbf{x})$ is not a piecewise constant function while the estimator function is. 
For a given tree $\mathbb{T},$ the corresponding estimator $\widehat{\bm\theta}(\mathbf{x}|\mathbb{T})$ is defined as
\[
\widehat{\bm\theta}(\mathbf{x}|\mathbb{T})=\sum_{\ell=1}^K \widehat{\theta}_\ell\mathbf{1}_{\mathbf{x}\in \mathcal{T}_\ell},
\]
where
\[
\widehat{\theta}_\ell = \arg \max_{\bm\theta\in \Theta} \frac{1}{n}\sum_{i=1}^n \log c_{\bm\theta}(\widehat{\mathbf{U}}_i)\mathbf{1}_{\mathbf{X}_i\in \mathcal{T}_\ell},
\]
and
$(\widehat{\mathbf{U}}_i)_{1\leq i \leq n}$ are pseudo-observations, that is estimated versions of $(\mathbf{U}_i)_{1\leq i \leq n}.$

Typically, these pseudo-observations are the result of a preliminary estimation of the marginal distribution, namely $\widehat{U}_i^{(j)}=\widehat{F}^{(j)}(Y^{(j)}_i|\mathbf{X}_i),$ but alternative procedures are possible: for example, a parametric model can be used to handle the margins. In Section~\ref{sec_margins}, we also discuss the possibility of relying on tree-based methods to estimate the margins as well, although there is no obligation to use the same type of technique for the dependence structure as for the margins. Therefore, in the following, we will try to keep our results as general as possible, expressing convergence conditions that this step should verify, but without imposing a specific method. However, let us point out that an interesting feature of regression trees is their ability to deal with both quantitative and qualitative covariates, which requires relying on estimation techniques for the margins that satisfy the same requirements.

The rest of the section is devoted to presenting our estimation procedure based on regression trees.
We describe the CART procedure consisting of two steps:  the construction of the maximal tree (Section~\ref{sec_max_tree}) which determines the proper decomposition of the covariate space $\mathcal{X}$ to obtain the regression tree $\mathbb{T}$ and deduce an estimator $\widehat{\bm\theta}(\cdot|\mathbb{T})$, and the pruning step (Section~\ref{sec_pruning}) which corresponds to a selection model step. Fitting the dependence structure requires a preliminary estimation of the margins, which is done once and for all before starting the algorithm. Various methods may be used to deal with thus preliminary step, the only requirement being that they satisfy the conditions under which our theoretical results hold. Examples of possible methods to estimate the margins are presented in Section~\ref{sec_margins}.

\subsection{Regression tree estimation of the dependence structure}
\label{sec_cart}


Regression trees provide an easy and transparent way to group observations that have similar behavior in terms of the response variable $Y$. They constitute a nonparametric regression model capable of reproducing highly nonlinear trends in the data and are thus able to approximate a wide class of functions. In addition, it can include both quantitative and categorical (non-ordinal) covariates.

Originally proposed by \citet{breiman}, regression trees are implemented through the CART algorithm, which involves a two-step process. Initially, a maximal tree is constructed, forming a binary structure that assigns observations to numerous classes (leaves), often leading to overfitting. Subsequently, the maximal tree is pruned to identify the subtree that offers the best compromise between complexity and generalization ability.

Section~\ref{sec_max_tree}  describes how the construction of the optimal tree takes place, making explicit our split criterion based on the maximization of the log-likelihood of the copula mixture model. In Section~\ref{sec_pruning}, we define the penalization criterion and discuss the pruning phase.




\subsubsection{Construction of the maximal tree}
\label{sec_max_tree}
Recall that, as mentioned before, the computation of the pseudo-observations $\widehat{\mathbf U}_i$ is done once and for all before starting the algorithm. The CART procedure is applied to $(\widehat{\mathbf U}_i, \mathbf X_i)_{1\leq i\leq n}$with the aim of maximize the log-likelihood of the model described in Section~\ref{sec_notation}. Such log-likelihood function can be written as the sum of the log-likelihoods of the parametric copulas estimated for the individual leaves of the tree:
\begin{equation*}
    \mathcal{L}_n (\theta_1,\ldots,\theta_K) = \sum_{\ell=1}^K \bigl( \frac{1}{n}\sum_{i=1}^n \log c_{\theta_\ell}(\widehat{\mathbf{U}}_i)\mathbf{1}_{\mathbf{X}_i\in \mathcal{T}_\ell} \bigr).
\end{equation*}
More precisely, the log-likelihood of the model is maximized conditionally on the covariates as a consequence of the recursive partitioning of the observations. In fact, at each split the observations are separated by looking at their values for one of the covariates. Formally, if we denote by $D_P=(\widehat{\mathbf{U}}_i,\mathbf X_i)_{i \in P}$ the observations that belong to a certain node $P$ (parent) and by $R_P(\mathbf X)$ the condition of the covariates that identifies those observations---such that $R_P(\mathbf X_i)$ is $1$ if $(\widehat{\mathbf{U}}_i,\mathbf X_i) \in D_P$, $0$ otherwise - then the left and right child nodes are determined by conditions of the type $\{ X_i^{(j)} \leq s, (\widehat{\mathbf{U}}_i,\mathbf X_i) \in D_P \}$ and $\{ X_i^{(j)} > s , (\widehat{\mathbf{U}}_i,\mathbf X_i) \in D_P \}$. Such split is uniquely determined by the pair $(j, s)$, with $j=1,\dots,p$ and $s\in \mathbb{R}$. (This is true for quantitative covariates, while in presence of qualitative covariates, the split is performed following Remark~\ref{r_qual}.)

Initially, all observations are in the root of the tree, implying that the dependence among variables $\widehat U_i$ is modeled by a single copula with parameters $\bm\theta_{root}^0$, which are optimized using maximum likelihood estimation (MLE). Subsequently, each split is carried out to maximize the increase in model log-likelihood. This gain simply corresponds to the sum of log-likelihoods estimated for the child nodes—once more, evaluated in correspondence of the parameters optimized via MLE—from which log-likelihood of the parent node is subtracted. In practical terms, the optimal gain, and thus the optimal split, is determined by testing all possible splits. The splitting process ends when further splits fail to enhance the log-likelihood, or more commonly, upon meeting specific stopping criteria. For instance, a common criterion is setting a minimum number of observations per leaf node.

The pseudocode summarizing the construction of the maximal tree is presented in Algorithm~\ref{alg:cart}.

{\tiny
\RestyleAlgo{ruled}
\SetKwComment{Comment}{\#}{}
\begin{algorithm}
    \footnotesize
    \caption{Construction of the maximal tree}\label{alg:cart}
    \SetAlgoLined
    \DontPrintSemicolon
    \SetKwData{Rroot}{$R_{root}$}
    \SetKwData{ListRulesInternalNodes}{ListRulesInternalNodes}
    \SetKwData{ListRulesLeaves}{ListRulesLeaves}
    \SetKwData{D}{D}
    \SetKwData{DP}{$D_P$}
    \SetKwData{DL}{$D_\ell$}
    \SetKwData{DR}{$D_R$}
    \SetKwData{RP}{$R_P$}
    \SetKwData{RL}{$R_\ell$}
    \SetKwData{RR}{$R_R$}
    \SetKwData{gain}{gain}
    \SetKwData{bestgain}{best\_gain}
    \SetKwData{bestj}{best\_j}
    \SetKwData{bests}{best\_s}
    \SetKwData{true}{true}
    \SetKwData{false}{false}
    \SetKwFunction{BuildTree}{BuildTree}
    \SetKwFunction{StoppingCriteria}{StoppingCriteria}
    \SetKwFunction{FindOptimalSplit}{FindOptimalSplit}
    \SetKwFunction{AddItem}{AddItem}
    \SetKwFunction{RemoveItem}{RemoveItem}
    \SetKwFunction{size}{size}
    \SetKwFunction{LogL}{LogL}
    \SetKwProg{function}{function}{:}{}

    \KwData{
        \D $\gets (\mathbf{X}_i,\mathbf{\widehat{U}}_i)_{i=1,\dots,n}$
    }\;
    \function{\StoppingCriteria{\D}}{
        \Comment{Define conditions to stop tree growth}
        \Comment{For example limit the minimum observations per leaf}
        \KwRet{\true if stopping criteria met, otherwise \false}
    }\;
    \function{\FindOptimalSplit{\DP}}{
        \Comment{Initialization}
        $\bestgain, \bestj, \bests \gets (0, -999, -999)$\;
        \BlankLine
        \Comment{Grid search over all the possible features and split values}
        \For{each possible $(j, s)$}{
            $\DL \gets \DP[X^j \leq s]$\;
            $\DR \gets \DP[X^j > s]$\;
            $\gain \gets \LogL{\DL} + \LogL{\DR} - \LogL{\DP}$\;
            \If{\gain > \bestgain}{
            $\bestgain, \bestj, \bests \gets (\gain, j, s)$\;
            }
        }
        \BlankLine
        \KwRet{(\bestj, \bests)}
    }\;
    \function{\BuildTree{\D}}{
        \Comment{Initialization}
        \Rroot $\gets \{\forall \mathbf{X}\}$\;
        \ListRulesInternalNodes $\gets [\Rroot]$\;
        \ListRulesLeaves $\gets [\quad]$\;
        \BlankLine
        \Comment{Tree construction}
        \While{\size{\ListRulesInternalNodes}>0}{
            \Comment{Retrieve the rule and the observations of the parent node to be split}
            \RP $\gets$ \ListRulesInternalNodes[0]\;
            \DP $\gets$ \D[\RP{\D}]\;
            \BlankLine
            \eIf{\StoppingCriteria{\DP}}{
                \Comment{Move the rule defining this node in the list of the leaves}
                \ListRulesInternalNodes $\gets$ \RemoveItem{\ListRulesInternalNodes, \RP}\;
                \ListRulesLeaves $\gets$ \AddItem{\ListRulesLeaves, \RP}\;
            }{
                \Comment{Find the optimal split and compute the rules defining the left/right children} 
                $(j^*, s^*) \gets$ \FindOptimalSplit{\DP}\;
                \RL $\gets$ \RP $\wedge \{X^{j^*} \leq s^*\}$\;
                \RR $\gets$ \RP $\wedge \{X^{j^*} > s^*\}$\;
                \BlankLine
                \Comment{Replace the rule of the parent node with the ones of its children}
                \ListRulesInternalNodes $\gets$ \RemoveItem{\ListRulesInternalNodes, \RP}\;
                \ListRulesInternalNodes $\gets$ \AddItem{\ListRulesInternalNodes, \RL}\;
                \ListRulesInternalNodes $\gets$ \AddItem{\ListRulesInternalNodes, \RR}\;
            }
        }
        \BlankLine
        \KwRet{\ListRulesLeaves}
    }\;
\end{algorithm}
}

\begin{remark}
\label{r_qual}
The implementation of the CART algorithm here proposed as example requires quantitative (or binary) covariates, for which an ordering of values is straightforward. For qualitative variables with $M>2$ modalities, the algorithm should include an ordering step preliminary to the split research, as suggested in \citep{rpart_intro}. Specifically, first the modalities are sorted by increasing values of the $\widehat{\theta}$ parameter estimated by considering the observations associated with each modality. Then, the $M-1$ possible splits are evaluated and the optimal one identified. An example of this procedure is available in the code we implemented for the application on the human influenza data (\ref{sec:real_data_example}).
\end{remark}

\subsubsection{Pruning step}
\label{sec_pruning}

Obtaining the maximal tree from the CART algorithm is not sufficient to have a proper estimation of the objective function $\bm{\theta}^0$, since this decomposition leads to overfitting. A subtree must be extracted from this maximal tree. This subtree will achieve a proper compromise between goodness of fit and complexity.

The complexity is here measured in terms of the number of leaves of a given tree $\mathbb{T}$. The selected subtree is thus obtained through the maximization of the following penalized criterion,
\begin{equation}
\label{penalized}
\bar{\bm\theta}(\mathbf{x})=\arg \max_{\widehat{\bm\theta}(\cdot|\mathbb{T})}\frac{1}{n}\sum_{i=1}^n \log c_{\widehat{\bm\theta}(\mathbf{X}_i|\mathbb{T})}(\widehat{\mathbf{U}}_i)-\lambda dim(\mathbb{T}).
\end{equation}
where the $\arg\max$ is taken over all subtrees $\widehat{\bm\theta}(\cdot|\mathbb{T})$ of $\mathbb{T}$ and $dim(\mathbb{T})$ is the number of leaves of $\mathbb{T}.$
This criterion could give the impression that one needs to consider all the possible subtrees within the maximal tree, and then select the optimal one. Fortunately, the particular shape of the penalty in \eqref{penalized} ensures that the best tree with $K$ leaves (according to this criterion) is a subtree of the best tree with $K+1$ leaves \citep[see][]{breiman}. This selection is then performed through validation on a test sample or cross-validation.


\subsection{Estimation of the margins}
\label{sec_margins}

Let us consider a given margin $Y^{(j)}$. If the components of $\mathbf{X}$ are all continuous covariates, a simple non-parametric estimator can be obtained via, for example, kernel smoothing. Following \citet{gijbels2012multivariate}, that is
\begin{equation}\label{kernelF}\widehat{F}^{(j)}(t^{(j)}|\mathbf{x})=\frac{\sum_{i=1}^n K\left(\frac{\mathbf{X}_i-\mathbf{x}}{h}\right)\mathbf{1}_{Y_i^{(j)}\leq t^{(j)}}}{\sum_{i=1}^n K\left(\frac{\mathbf{X}_i-\mathbf{x}}{h}\right)},
\end{equation}
where $h>0$ is the bandwidth and  the kernel $K$ is, for example, $K(\mathbf{x})=\prod_{j=1}^d k(x^{(j)})$ with $k$ a positive function and such that $\int k(u)\mathrm{d}u=1.$ As it is classical for kernel estimators, the rate of uniform convergence is $O(h^2+[\log n]^{1/2}n^{-1/2}h^{-d/2})$ \citep[see e.g.][]{einmahl2005uniform}, where $h^2$ corresponds to the bias term.

On the other hand, this estimator is not valid if $\mathbf{X}$ contains some qualitative components. In this perspective, consider the case where $\mathbf{X}$ has $M$ modalities,  a possible non-parametric estimator is then
\[
\widehat{F}^{(j)}(t^{(j)}|\mathbf{x})=\frac{1}{n_{\mathbf x}}\sum_{i=1}^n \mathbf{1}_{Y_i^{(j)}\leq t^{(j)}}\mathbf{1}_{\mathbf{X}_i=\mathbf{x}},
\]
where $n_{\mathbf x}=\sum_{i=1}^{n}\mathbf{1}_{\mathbf{X}_i=\mathbf{x}}.$  Note that since the covariates are assumed to be random, so is $n_{\mathbf x}$. If $n_{\mathbf x}$ were not random, the rate of convergence would be typically the same as for an empirical c.d.f., that is $n_{\mathbf x}^{-1/2}.$ 

However, this approach quickly reaches its limits since, when $M$ is large, the number of observations such that $\mathbf{X}_i=\mathbf{x}$ becomes quite small, diminishing considerably the rate of convergence. An alternative is to build classes of modalities, that is decomposing the set of covariates into $m<M$ modalities, as is the case if regression trees are also applied to the margins. Consider this decomposition into $m$ modalities, and let $\mathcal M(\mathbf x)$ denote the class to which x belongs to, and 
\[
\widehat{F}^{(j)}(t^{(j)}|\mathbf{x})=\frac{\sum_{i=1}^n \mathbf{1}_{Y_i^{(j)}\leq t^{(j)}}\mathbf{1}_{\mathbf{X}_i\in \mathcal{M}(\mathbf{x})}}{\sum_{i=1}^n \mathbf{1}_{\mathbf{X}_i\in \mathcal{M}(\mathbf{x})}}.
\]
In this case, a bias term appears, since $\widehat{F}^{(j)}(t^{(j)}|\mathbf{x})$ converges towards $\mathbb{P}(Y^{(j)}\leq t^{(j)}|\mathcal{M}(\mathbf{X})=\mathcal{M}(\mathbf x)).$

An alternative way to proceed is also to consider a parametric model for the margins, like, for example, Generalized Linear Model \citep[see e.g.][]{nelder1972generalized,hastie2017generalized}. In this case, and under proper assumptions, the convergence rate in the estimation of the margins can become $n^{-1/2},$ up to a strong assumption on the distributions.

Due to the variety of possible approaches in estimating the margins, we will keep the rest of the paper as general as possible regarding this point, only requiring generic convergence assumptions on this preliminary step.

\section{Consistency results}
\label{sec2}

This section is dedicated to presenting the main theoretical results that validate the asymptotic behavior of the copula tree estimation procedure. We gather and discuss the list of assumptions required to obtain these results in Section~\ref{sec:assum}. Moving on to Section~\ref{sec:cons}, we explore the consistency of a single tree (with a given number of leaves $K$ that may tend to infinity). We chose to focus on the stochastic part of the error, while the approximation error is expected to decrease with $K.$ Rhe rate of the decrease depends on the specific shape of the target function $\theta^0(\mathbf{x})$ which remains an open problem in the regression tree literature. In Section~\ref{sec:oracle}, we investigate the ability of the penalized criterion to achieve a similar performance as if  the optimal number of leaves were known.

\subsection{Conditions and assumptions}
\label{sec:assum}

First, Assumption~\ref{a:pseudo} below controls the rate of consistency of the pseudo-observations, which represents the c.d.f. of the margins.

\begin{assum}
\label{a:pseudo}
 Assume that
\begin{equation} \label{a11}
\sup_{\substack{i=1,\ldots,n\\{j=1,\ldots,k}}}\left|\frac{U^{(j)}_i}{\widehat{U}_i^{(j)}}+\frac{1-U^{(j)}_i}{1-\widehat{U}^{(j)}_i}\right| = O_P(1) \, .
\end{equation}
For some $0<\alpha<1/2,$
\begin{equation}
\label{a12}
\sup_{\substack{{i=1,\ldots,n}\\{j=1,\ldots,k}}}\left|\frac{\widehat{U}^{(j)}_i-U^{(j)}_i}{\left[U_i^{(j)}(1-U_i^{(j)})\right]^{\alpha}}\right|=O_P(\varepsilon_n),
\end{equation}
for some sequence $\varepsilon_n$ that tends to 0 as $n$ tends to infinity.
\end{assum}

Assumption~\ref{a11} is here to control the behavior of the pseudo-observations near the border of $[0,1]^k.$ If the margins are estimated via empirical distribution functions, this assumption easily holds from \citep[see][Remark ii]{wellner1978limit}. In case this assumption would not hold for more complex estimators, it can be avoided through the introduction of trimming, that is removing points too close to the boundaries of the unit square. This would introduce a bias that can be controlled through \eqref{a12} (see remark \ref{rem_trimming}).

As it will appear in the theoretical results of Sections~\ref{sec:cons} and~\ref{sec:oracle}, this rate is expected to go faster to zero than the part related to the estimation of the tree itself, otherwise it will be predominant.It is important to note that \eqref{a12} is similar to the slightly stronger condition
\[
\sup_{\substack{{t \in \mathbb R}\\{j=1,\ldots,d}}}\left|\frac{\widehat{F}^{(j)}(t|\mathbf{x})-F^{(j)}(t|\mathbf{x})}{\left[F^{(j)}(t|\mathbf{x})(1-F^{(j)}(t|\mathbf{x}))\right]^{\alpha}}\right|=O_P(\varepsilon_n).
\] 
If we consider the estimation of the (unconditional) c.d.f. $F^{(j)}(t)=\mathbb{P}(Y^{(j)}\leq t)$ by the empirical distribution function, this condition is easily satisfied  with $\varepsilon_n=n^{-1/2}$ \citep[see][Example 19.12]{vandervaart}. In the case of the kernel-based estimator, Section~\ref{sec:margins1} shows that the rate is slower, namely $\varepsilon_n=(h^{2}+[\log n]^{1/2}n^{-1/2}h^{-d/2})$, and, in the case of discrete covariates, Section~\ref{sec:margins2} shows that $\varepsilon_n=n^{-1/2}.$

Before presenting the rest of the assumptions, we introduce two conditions on classes of functions which will be necessary in the following.

\begin{cond}\label{cond:1}
A class of functions $\mathcal{F}=\{\mathbf{u}\mapsto \varphi_{\bm\theta}(\mathbf{u}):\bm\theta\in \Theta\}\subset L^2(\mathbb{R}^k)$ (for some $k> 0$) is said to satisfy Condition~\ref{cond:1} if
\[
|\varphi_{\bm\theta}(\mathbf u)-\varphi_{\bm\theta'}(\mathbf u)|\leq B(\mathbf{u})\|\bm\theta-\bm\theta'\|_{1}, \quad \mathbf{u} \in [0,1]^k
\]
where $B$ is a function in $\mathbb{R}^k$ such that $\mathbb{E}[B(\mathbf{U})^2]<\infty.$ 

For such a class, there exists an envelope function, that is a function $\Phi$ such that, for all $\bm\theta\in \Theta,$ $|\phi_{\bm\theta}(\mathbf{u})|\leq \Phi(\mathbf{u})$ and $\mathbb{E}[\Phi(\mathbf{U})^2]<\infty.$ Taking any point $\widetilde{\bm\theta}\in \Theta,$ $\Phi$ can be chosen as $\Phi(\mathbf{u})=\varphi_{\widetilde{\bm\theta}}(\mathbf{u})+diam(\Theta)B(\mathbf{u}),$ where $diam(\Theta)$ denotes the diameter of the compact set $\Theta.$
\end{cond}

\begin{cond}\label{cond:2}
A class of functions $\mathcal{F}=\{\mathbf{u}\mapsto \varphi_{\bm\theta}(\mathbf{u}):\bm\theta\in \Theta\}\subset L^2(\mathbb{R}^k)$ (for some $k> 0$) is said to satisfy Condition~\ref{cond:2} if
\begin{enumerate}
    \item 
there exist an envelope $\Phi$ and a universal constant $A_1$ such that, for all $\varphi\in \mathcal{F},$
\[
|\varphi(\mathbf{u})|\leq \Phi(\mathbf{u})\leq A_1\sum_{r=1}^k\frac{1}{\{u^{(r)}[1-u^{(r)}]\}^{\beta_1}}, \quad \mathbf{u}\in [0,1]^k,
\]
with $0\leq \beta_1<1/2.$
\item there exists a universal constant $A_2$ such that for all $\varphi\in \mathcal{F},$ 
\[
|\partial_j \varphi(\mathbf{u})|\leq \frac{A_2}{\left\{u^{(j)}[1-u^{(j)}]\right\}^{\beta_2}} \sum_{r=1}^d \frac{1}{\{u^{(r)}[1-u^{(r)}]\}^{\beta_3}}, 
\]
with $0\leq \beta_2\leq 1$, $\beta_3<1/2$, and
where $\partial_j$ denotes the partial derivative with respect to the $j-$th component of $\mathbf{u}.$
\end{enumerate}
\end{cond}
Theses conditions allow controlling the complexity of the class of functions and are related to classical assumptions used for the consistency of classical maximum likelihood estimators \citep[see e.g.][]{vandervaartwellner}.
The second condition is required to control the behavior of the copula log-likelihood and of its derivatives close to the boundaries of $[0,1]^d.$ 
These conditions are similar to the one used for example in \citep{omelka,tsukahara}. They hold for many classical classes of copula functions like Gaussian, Clayton, Frank, Gumbel families. We thus consider the two following assumptions.

\begin{assum}\label{assum:llk}
Let 
\[
\mathcal{F}_1 = \left\{\mathbf{u}\mapsto \log c_{\bm\theta}(\mathbf{u})\, , \bm\theta\in \Theta \right\}.
\]
Assume that $\mathcal F_1$ satisfies Conditions~\ref{cond:1} and \ref{cond:2}.
\end{assum}

\begin{assum}\label{assum:grad:llk}
Let 
\[
 \mathcal{F}_2 = \left\{\mathbf{u}\mapsto \nabla_{\bm\theta} \log c_{\bm\theta}(\mathbf{u})\, , \bm\theta\in \Theta \right\}.
\]
Assume that $\mathcal F_2$ satisfies Conditions~\ref{cond:1} and \ref{cond:2}.
\end{assum}

\subsection{Asymptotic theory for a single tree}
\label{sec:cons}

In this section, we consider a tree $\mathbb{T}=(\mathcal{T}_\ell)_{\ell=1,\ldots,K}$ with $K$ leaves. 

Let 
\[
\bm\theta^0 = (\theta^0_1,\ldots,\theta^0_K)=\arg \max_{(\theta_1,\ldots,\theta_K)}\sum_{\ell=1}^K \mathbb{E}\left[\log c_{\theta_\ell}(\mathbf U)\mathbf{1}_{\mathbf{X}\in \mathcal{T}_\ell}\right],
\]
where the maximum is supposed to be achieved at a unique point $(\theta^0_1,\ldots,\theta^0_K),$ and we denote
\[
\bm\theta^0(\mathbf{x}|\mathbb{T})=\sum_{\ell=1}^K \theta^{0}_\ell\mathbf{1}_{x\in \mathcal{T}_\ell}.
\]

Proposition~\ref{prop_consistency} presented below is a consistency result. To that purpose, we consider the $L^1$-norm to compare our maximum likelihood estimator $\widehat{\bm\theta}(\cdot|\mathbb{T})$ and $\bm\theta^0(\cdot|\mathbb{T})$:

\[
\|\widehat{\bm\theta}(\cdot|\mathbb{T})-\bm\theta^{0}(\cdot|\mathbb{T})\|_1=\int \|\widehat{\bm\theta}(\mathbf{x}|\mathbb{T})-\bm\theta^0(\mathbf{x}|\mathbb{T})\|_1\mathrm{d}\mathbb{P}_{\mathbf{X}}(\mathbf{x})=\sum_{\ell=1}^K | \widehat \theta_\ell - \theta^0_\ell| \mathbb{P}(\mathbf X \in \mathcal T_\ell),
\] 
where $\mathbb{P}_{\mathbf{X}}$ is the distribution of the covariates $\mathbf{X}$. 

\begin{proposition}
\label{prop_consistency}
Under Assumptions~\ref{a:pseudo} to~\ref{assum:llk}, and if $n[K\log K]^{-1}\rightarrow \infty,$
\[
\|\widehat{\bm\theta}(\cdot|\mathbb{T})-\bm\theta^{0}(\cdot|\mathbb{T})\|_1=o_P(1).
\]
\end{proposition}

By considering an additional assumption on the copula family, namely Assumption~\ref{assum:grad:llk}) and conditions on the Hessian matrix, we obtain the convergence rate.

\begin{theorem}
\label{th_main} Under Assumptions~\ref{a:pseudo} to~\ref{assum:grad:llk}, and assume furthermore that for $\ell=1,\ldots,K,$ the Hessian matrix
$\nabla^2_{\theta}\log c_{\bm\theta}(\theta^0_\ell)$ is invertible. Then, 
\[
\|\widehat{\bm\theta}(\cdot|\mathbb{T})-\bm\theta^{0}(\cdot|\mathbb{T})\|_1=O_P\left(\frac{[K\log K]^{1/2}}{n^{1/2}}+\varepsilon_n\right).
\]
\end{theorem}

It is not surprising to notice that the stochastic part of the error deteriorates with $K,$ due to the increase of the complexity of the model. On the other hand, although this part is harder to track, the approximation error is supposed to decrease with $K,$ which means that $\bm\theta^0(\cdot|\mathbb{T})$ is supposed to be closer to the "true" target function $\bm\theta^0(\cdot)$ when the number of leaves of $\mathbb{T}$ increases.

\subsection{Oracle property for the pruning step}
\label{sec:oracle}

Let us define the optimal subtree extracted from the maximal tree $\mathbb{T}_{\max}$ (which has $K_{\max}$ leaves) as
\begin{equation}
\bm\theta^0(\mathbf{x})=\arg \max_{\bm\theta^0(\cdot|\mathbb{T})}\mathbb{E}\left[ \log c_{\bm{\theta}^0(\mathbf{X}|\mathbb{T})}({\mathbf{U}})\right].
\end{equation}
Let $K^0$ denote the number of leaves of $\bm\theta^0.$ If $K^0$ were known, Theorem~\ref{th_main} shows that one may expect a convergence rate of $\sqrt{K^{0}\log K^{0}/n}$ for the stochastic part. The next result shows that the penalized procedure has the ability to asymptotically achieve this optimal rate even though the number $K^{0}$ is unknown. This, of course, requires conditions on the penalizing constant $\lambda.$

\begin{theorem}
\label{th_main2}
Assume that the assumptions of Theorem~\ref{th_main} hold for all of the subtrees of the maximal tree with $K_{\max}$ leaves.
Then if $\lambda\rightarrow 0,$ and if $\lambda n^{1/2}[K_{\max}\log K_{\max}]^{-1/2}\rightarrow \infty,$
\[
\|\bar{\bm\theta}-\bm\theta^0\|_1=O_P\left(\frac{[K^0\log K^0]^{1/2}}{n^{1/2}}+\varepsilon_n\right).
\]
\end{theorem}

All the proofs are postponed to the Appendix section.

\section{Empirical evidence}
\label{sec4}

\subsection{Simulation study}
In this section, we present the functioning of the conditional copula analysis on simulated data.

\subsubsection{The regression framework}
We consider a bivariate random variable $\mathbf{U} = (U^{(1)},U^{(2)})$, with uniform margins over $[0,1]$ and  distributed according to an Archimedean copula $C_{\theta(\mathbf{X})}.$ Archimedean copulas is a standard family of copulas often used in modeling applications, \citep[see][]{kularatne2021use, hennessy2002use}. They are determined by a single parameter $\theta \in \mathbb{R}^1$ which is associated with Kendall's $\tau$ coefficient through a bijective relationship \citep{genest}. In our framework, $\theta$, and thus also $\tau$, depends on two covariates $\mathbf{X} = (X^{(1)}, X^{(2)})$, which are random variables uniformly distributed in $[0,1]$. Moreover, we assume that $\mathbf{U}_i$ are  samples of the true cumulative marginal distributions of some bivariate response variables $\mathbf{Y} = (Y^{(1)}, Y^{(2)})$ conditionally on the covariates $\mathbf{X}$. Specifically, we assume normal distributions for these margins, with the mean parameters being a linear function of $\mathbf{X}$.  

The first step of the simulations consists of generating synthetic data for $(\mathbf{X}_i, \theta_i, \tau_i, \mathbf{U}_i, \mathbf{Y}_i)$. Hence, our goal is to estimate the parameters $\theta_i$ (or $\tau_i$) from $(\mathbf{X}_i, \mathbf{Y}_i)$, pretending not to know the true observations $\mathbf{U}_i$, as it is usually the case in a real data scenario. Therefore, as a preliminary step to the conditional copula analysis, we first compute the pseudo-observations. We do that by considering two different approaches, a parametric and a non-parametric one, which will result in two vectors of pseudo-observations, namely $\mathbf{V}$ and $\mathbf{W}$. Eventually, we fit the conditional copula model to both the $\mathbf{V}$ and $\mathbf{W}$ pseudo-observations, other than to the true margins $\mathbf{U}$ for additional comparison. The goodness of the three fits is evaluated against a benchmark model.

\paragraph{Definition of different scenarios}

 
To investigate different scenarios, we consider three Archimedean copulas - the Clayton, Frank, and Gumbel copulas. We also consider three types of dependence between $\tau$ and the covariates $(X^{(1)}, X^{(2)})$, which we report below: 
\begin{itemize}
    \item[(i)] a step-wise function:
    \[
    \tau_i = 
    \begin{cases}
        0.3 & \textrm{if } X^{(1)}_i < 0.4, X^{(2)}_i < 0.75 \\
        0.5 & \textrm{if } X^{(1)}_i \geq 0.4, X^{(2)}_i < 0.75 \\
        0.7 & \textrm{if } X^{(1)}_i < 0.4, X^{(2)}_i \geq 0.75 \\
        0.9 & \textrm{if } X^{(1)}_i \geq 0.4, X^{(2)}_i \geq 0.75
    \end{cases}
    \]
    \item[(ii)] a steep sigmoid:
    \[
    \tau_i = 0.3 - \frac{0.2}{1+\exp(-40(X^{(1)}_i - 0.4))} - \frac{0.4}{1+\exp(-40(X^{(2)}_i - 0.75))}
    \]
    \item[(ii)] a gentle sigmoid:
    \[
    \tau_i = 0.3 - \frac{0.2}{1+\exp(-15(X^{(1)}_i - 0.4))} - \frac{0.4}{1+\exp(-15(X^{(2)}_i - 0.75))}
    \]
\end{itemize}

With these constraints, having fixed the covariates $(X^{(1)}, X^{(2)})$, we obtain nine different conditional copulas, from which we sample $\mathbf{U}$ observations. Let us specify that these conditional copulas are defined such that Kendall's $\tau$ coefficients always range in the interval $[0.3, 0.9]$, to ensure comparability.

Finally, in all scenarios the response variables $\mathbf{Y}$ is defined as follow:
\[
\begin{cases}
        Y_i^{(1)} & =  \Psi^{-1}(U^{(1)}_i - 1-0.2X^{(1)}_i-0.05X^{(2)}_i)\\
        Y_i^{(2)} & =  \Psi^{-1}(U^{(2)}_i-1+0.1X^{(1)}_i-0.2X^{(2)}_i )\\
    \end{cases}
\]
where $\Psi$ is the c.d.f. of the distribution $\mathcal N(0,1)$. 

\paragraph{Pseudo-observation computation}
We consider two alternative methods to compute the pseudo-observations. 

First, in a parametric approach, we assume that the marginal distributions of $Y^{(j)}$ conditionally on $\mathbf{X}$ can be approximated by normal distributions with variance fixed at 1. Thus, we estimate the mean parameter through a linear model, i.e. $\hat{\mu}^{(j)}_i = LM(\mathbf{X}_i)$, and we compute the pseudo-observations $\mathbf{V}^{(j)}_i = \Psi^{-1}(Y^{(j)}_i- \hat{\mu}^{(j)}_i)$.

Second, to avoid assumptions on the form of the margins, we perform a kernel estimation depending on the covariates as defined in \eqref{kernelF}. We consider a simple Gaussian kernel, with the bandwidth $h$ optimized depending on the scenario, specifically, we used $h=0.4$ for Clayton and Frank copulas, $h=0.3$ for the Gumbel copula. This way, the pseudo-observations $\mathbf{W}_i$ are computed as empirical percentiles, where in the calculation of the empirical cumulative distribution function the different observations are weighted differently according to their distance in terms of covariates.

\paragraph{Model's evaluation}
As a reference model, we simply fit the Archimedean copula to $\mathbf{U}$ (and to $\mathbf{V}$ and $\mathbf{W}$), ignoring the additional information carried by the covariates. It means that we estimate a unique value for $\tau$, which corresponds to the estimation provided by the root of the regression tree of the conditional copula model. Hence, the prediction errors of the conditional copula model and of the benchmark model are compared. For comparison, we consider the Mean Squared Errors for both the estimates of the $\tau$ coefficients and the values of the cumulative copula and the log-likelihood values of the models toward the observations/pseudo-observations. 

\paragraph{Simulation results}

For each one of the nine settings presented above, we build 500 triples of datasets, containing 1000 observations $\mathbf{U}_i$, 1000 pseudo-observations $\mathbf{V}_i$, and 1000 pseudo-observations $\mathbf{W}_i$, respectively. Results are presented in Figure \ref{fig:fig1}. In all scenarios, the conditional copula model outperforms the benchmark model, both in terms of log-likelihood values and estimates for the $\tau$ coefficients and for the cumulative copula values.
As expected, the predictions worsen when the dependence on the covariates changes from a step function, which can be perfectly captured by a regression tree, to smoother functions. Finally, no significant changes are noticed when models are fitted to observations or pseudo-observations. We notice that the conditional copula model most of the time identifies five or six groups of observations. That corresponds to a slightly overfitting of the model, as four groups are expected.

\begin{figure}[!ht]
    \centering
    \includegraphics[width=\textwidth]{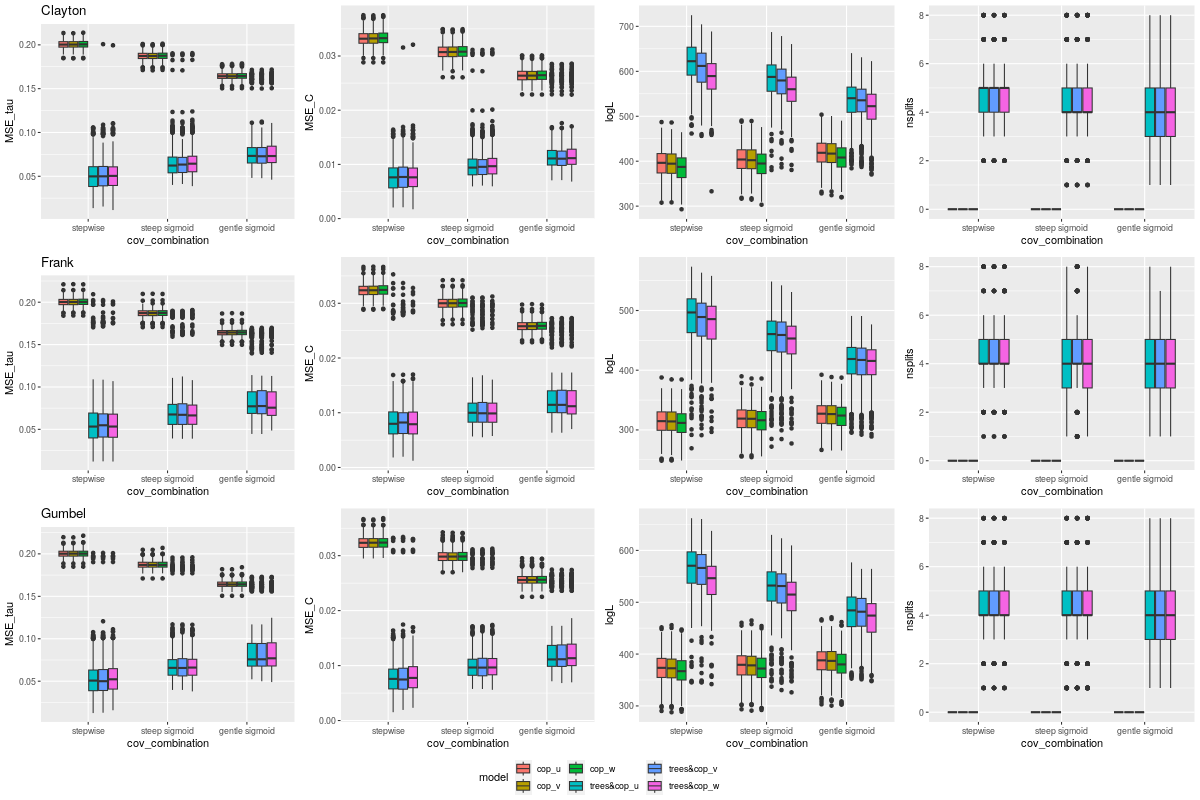}
    \caption{Results of simulations. Results for the Clayton, Frank, and Gumbel copulas are depicted on the different rows. For each copula, the results for the three types of covariate dependence are reported on the x-axis. The six colors identify different models: red, orange, and green are for the conditional copula model fitted on the observations $\mathbf{U}$ and on the pseudo-observations $\mathbf{V}$ and $\mathbf{W}$, respectively, while cyan, blue and magenta are for the benchmark model. In the first two columns, we show results in terms of MSE for the $\tau$ estimates and the cumulative copula estimates, in the third column in terms of log-likelihood. In the fourth column, we report the distributions of the number of splits, i.e. the number of leaves minus 1, identified by the regression trees of the conditional copula models. Each boxplot represents the results for 500 datasets of 1000 points each.}
    \label{fig:fig1}
\end{figure}

\subsection{Real data example}
\label{sec:real_data_example}
In this section, we present an application of the conditional copula model on epidemiological data of cases of human influenza.

\subsubsection{The human influenza: context and data}
\label{sec:flunet_data}
Three main influenza strains co-circulate worldwide and infect humans: influenza A/H1N1pdm, influenza A/H3N2, and influenza B. The relative proportions of the three viruses are highly variable in time and space and the unpredictability of the strains' (co-)dominance patterns poses a major limitation to the mitigation of the upcoming epidemic wave in terms of intervention design and vaccination. Here, we use the conditional copula model to capture some trends of the coupled dynamic of influenza subtypes. In particular, first, we assume that we can use Archimedean copulas to describe the dependence structure of the relative abundances of influenza subtypes across regions and years. Second, we implement the conditional copula model to identify spatio-temporal patterns of such dependence structure. 

The World Health Organisation provides data on influenza surveillance for several countries, consisting of weekly counts of cases classified by subtype \citep{gisrs2022, flahault1998flunet}. We consider data from 80 countries that reported a minimum of 50 classified cases per year in the period from April 2010 to April 2019. Then, we aggregate counts annually (from April to April) and for each country-year (800 observations in all) we compute the proportion of cases of A/H1N1pdm, A/H3N2, and B. We consider the relative abundances of subtypes as the response variables to be modeled with an Archimedean copula, testing Clayton, Frank, and Gumbel families, and the year and the Influenza Transmission Zone (ITZ) as the relevant covariates. The ITZs are 18 groups of countries with similar influenza transmission patterns identified by the W.H.O. \citep[for the precise definition of the groups see][]{who2018}.
Before fitting the conditional copula model, we perform a preprocessing step by applying an additive log-ratio transformation to the relative proportion of subtypes \citep{aitchison1982statistical}. This is a common procedure when working with percentage data \citep{jackson1997compositional, filzmoser2009univariate, buccianti2015exploring}, and it allows us to map bounded vectors (A/H1N1pdm \%, A/H3N2 \%, B \%)$ \in S^3$ into unbounded vectors $(Y^{(1)},Y^{(2)}) \in \mathbb{R}^2$, where $S^3$ is the so-called 3-part simplex. In particular, we use the isometric log-ratio transformation proposed by \citet{egozcue2003isometric}:
\[
\begin{cases} Y^{(1)} & = \sqrt{\frac{2}{3}} \ln{\frac{B\%}{\sqrt{A \backslash H1N1pdm\%  *  A \backslash H3N2\%}}} \\
Y^{(2)} & = \sqrt{\frac{1}{2}} \ln{\frac{A \backslash H1N1pdm\%}{A \backslash H3N2\%}} 
\end{cases}
\]
Thus, the actual response variable is $\mathbf{Y}=(Y^{(1)},Y^{(2)})$, with $Y^{(1)}$ describing the relative abundance between influenza B and the average proportion of influenza A subtypes, while $Y^{(2)}$ denotes the relative amount of A/H1N1pdm and A/H3N2.

\subsubsection{Model implementation}

\paragraph{Estimation of the margins.}


We consider regression trees to model the relationship between the response variables $Y^{(j)}$ and the covariates year and ITZ. 
Both the covariates are treated as categorical variables, meaning that a priori values have no precise sorting and an ordering step is needed preliminary to the split search, as explained in Remark~\ref{r_qual}. This allows maximum flexibility to the splitting procedure so that the tree can effectively capture the trends in the data.
Once the trees are optimized by cross-validation, the pseudo-observations $\widehat{\mathbf{U}}^{(j)}$ are computed from a mixture of empirical cumulative distribution functions defined over the groups of points identified by the optimal tree. That is 
\[
\widehat{\mathbf{U}}^{(j)}(t^{(j)}|\mathbf{X}) = \sum_{l=1}^{K} \left(\frac{1}{n_l} \sum_{i=1}^{n_l} \mathbf{1}_{Y^{(j)}_i \leq t^{(j)}} \mathbf{1}_{\mathbf{X}_i\in \mathcal{T}_l^{(j)}}\right),
\]
with $\mathcal{T}_{l=1,\dots,K}^{(j)}$ being the terminal nodes of the tree fitted on the $Y^{(j)}$ response variable.


\BlankLine
\paragraph{Fit of the conditional copula model.}
We test two-dimensional Clayton, Frank, and Gumbel copulas to model $\widehat{\mathbf{U}}$ conditionally on the covariates year and ITZ, again treated as categorical covariates. 
We implement a 3-fold cross-validation repeated 50 times to optimize the pruning of the trees with Breiman's rule used to identify the optimal tree. The best conditional model is the one with the Frank copula, which leads to the highest value of log-likelihood. It results in a tree with five leaves (Figure \ref{fig:fig2}) which will be discussed in the next paragraph.

\begin{figure}[!ht]
    \centering
    \includegraphics[width=0.6\textwidth]{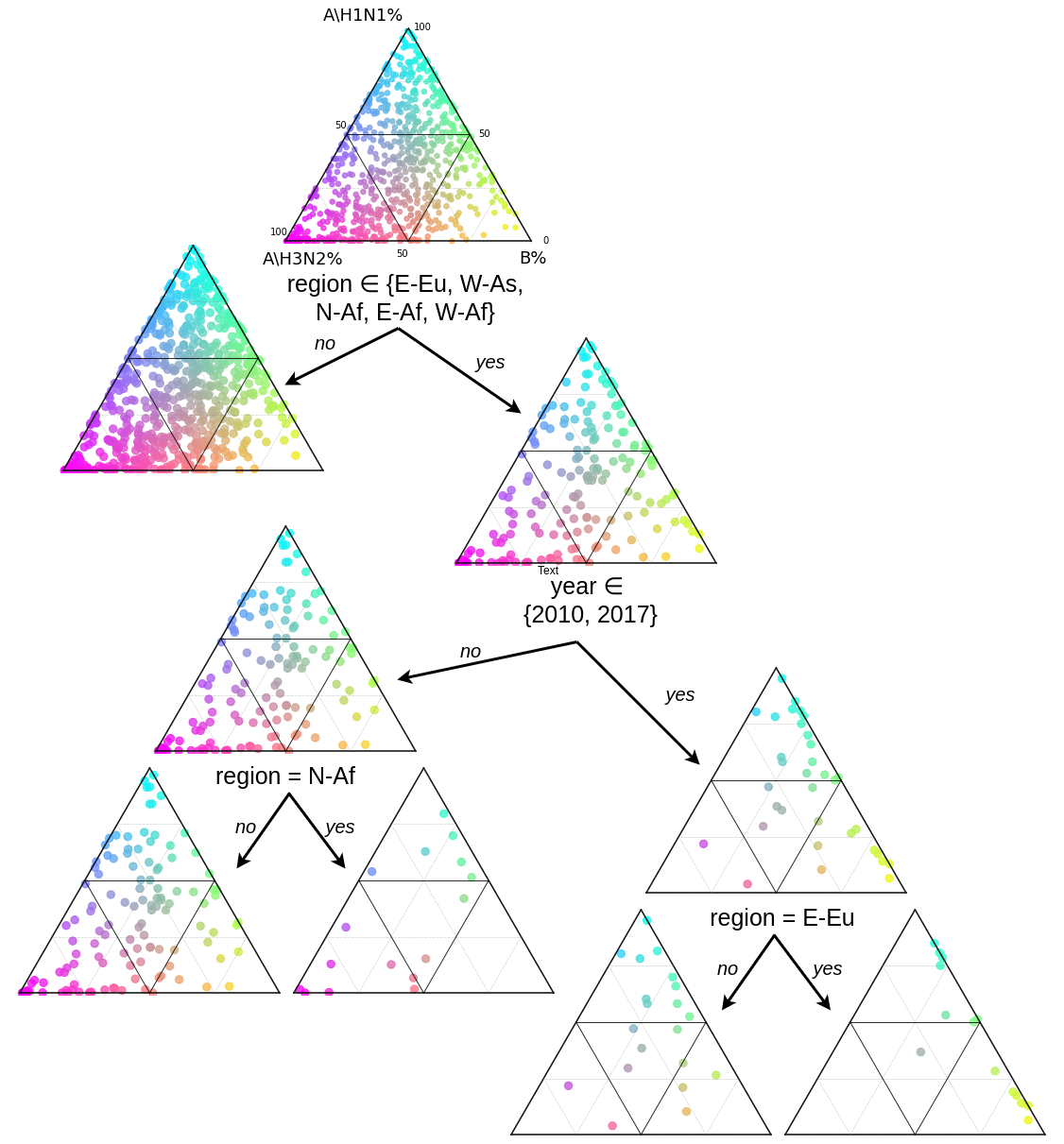}
    \caption{Optimal tree identified by the Frank conditional copula model applied to data of relative abundances of influenza subtypes across countries and regions. Data on the relative abundances of influenza subtypes are considered for 800 countries-years (corresponding to the 800 points in the top ternary plot). Similarly, for each node of the tree, a simplex represents the subtype relative abundances of the countries-years clustered in the node. We use a ternary color code to distinguish countries-years with dominance of A/H1N1pdm (cyan), A/H3N2 (pink), and B (yellow). 
    For each split, the condition used to partition the observations is indicated. 
    From top-left to bottom-right, the number of observations in each leaf is 630, 120, 16, 20 and 14. In the same order, Kendall's $\tau$ coefficients equal -0.06, -0.09, 0.3, -0.03, and 0.25.}
    \label{fig:fig2}
\end{figure}
\subsubsection{Results and discussions}

Thanks to the conditional copula model we were able to ameliorate the adjustment to the data; the log-likelihood of the Frank simple copula on all the 800 $\widehat{\mathbf{U}}$ pseudo-observations was 1.2, and it increased to 13.8 for the Frank copula mixture model identified by the optimal tree.

In the estimation of the response variables $Y^{(1)}$ and $Y^{(2)}$, the years are used more often than the regions to perform the splits, indicating that the relative abundances of B vs. A and A/H1N1pdm vs. A/H3N2, taken independently, varied more in time than in space (see Figures \ref{fig:tree_margins} in the Appendix). In other words, to a first approximation, we find consistent temporal dynamics worldwide, going a step further we also identify significant spatial patterns. 
It is interesting to note that each time the spatial information is used to perform the split, European regions are grouped together, sometimes with other neighboring regions (mainly North Africa and Western and Central Asia), and always separated from countries of the southern hemisphere. These spatial country groupings overall match well the geographical clustering found in other studies with different methods. Previous studies found evidence for an annual reseeding of influenza viruses from tropical and subtropical countries to temperate regions, especially for A/H3N2 viruses  \citep[see][]{bedford2015global, lemey2014unifying, bahl2011temporally, le2013migration}. These dynamics could also contribute to determining the patterns in subtype compositions that emerged from our analysis. However, our purely descriptive analysis does not allow us to speculate on any underlying mechanism. 

Once $\widehat{U}^{(1)}$ and $\widehat{U}^{(2)}$ are computed, the conditional copula model identifies significant changes in the pseudo-observations dependence across space and time. It results in a tree with five leaves, characterized by different degrees of correlation (Kendall's $\tau$ among the $\widehat{U}^{(1)}$ and $\widehat{U}^{(1)}$ pseudo-observations range from -0.09 to 0.3). 
However, we note that a single leaf contains most of the data points (630 out of 800), meaning that the simplifying assumption would probably provide a reasonable approximation for the majority of the countries-years in our analysis. However, the other leaves allow us to refine the fit of the data and further separate a few country years that are mainly characterized by a proportion of B infections higher than the average.

\section{Conclusion}

In this paper, we proposed a new methodology to model conditional copulas, based on regression trees. The technique requires the assumption that the conditional copulas all belong to the same family of parametric copulas, with the association parameter changing with the value of the covariates. The procedure presents many advantages. First, the tree structure theoretically allows capturing any form of the conditional association parameter. Second, the simplicity of the final model, if restricted to a single leaf of the tree, allows one to obtain a tractable output. We note that our approach allows a relaxation of the simplifying assumption \citep[see][]{derumigny2017tests}, but this remains valid for each of the subsets of data identified by the tree.
Another interesting feature is the ability to deal with discrete and/or continuous covariates. 

In addition, let us point out that this method can be easily extended to the case where several families of copulas are tested at each node. This would give a more complex final structure, since not only the association parameter but also the copula family could vary from one leaf to another.  However, it increases the complexity of the implementation of the algorithm. Finally, let us note that the potential weakness of the procedure is its instability. Like every regression tree procedure, our method can be very sensitive to new incoming data, as new information may considerably change the structure of the tree and the classes that are made. Careful attention should be given to this aspect. On the other hand, a direct extension that could reduce this instability would be to consider the corresponding random forest algorithm, that is, computing many small copula trees on separate bootstrap samples, and then aggregating them. The aggregation of these trees would be a way to stabilize the result, but of course, would reduce the interpretability of the model.

\paragraph{R codes:} The \texttt{R} codes are publicly available at \url{https://github.com/FrancescoBonacina/tree-based-conditional-copula-estimation}.

\section{Appendix}

The Appendix section is organized as follows. We first start with preliminary results that are needed to prove our results in Section~\ref{sec:prem:res}, including some results on the complexity of the class of functions defined by the model in Section~\ref{sec:bracket}, and section~\ref{sec:proof_pseudo} provides a general result that will be used repeatedly to handle deviations of the score function. We then prove  Proposition~\ref{prop_consistency} in Section~\ref{result_consistency}, Theorem~\ref{th_main} in Section~\ref{proof_th_main}, and Theorem~\ref{th_main2} in Section~\ref{proof_oracle}. Results on the convergence rates of the margins are then shown in Sections~\ref{sec:margins1} and~\ref{sec:margins2}.

\subsection{Preliminary results}\label{sec:prem:res}
In all this section, let us denote 
\[
\mathfrak{F}=\left\{(\mathbf{u},\mathbf{x}) \mapsto \phi(\mathbf{u};\mathbf{x})=\sum_{\ell=1}^K \varphi_\ell(\mathbf{u})\mathbf{1}_{\mathbf{x}\in \mathcal{T}_\ell} \text{ with for $\ell=1,\ldots, K$, } \varphi_\ell\in \mathcal{F} \text{ satisfying Condition~\ref{cond:1}} \right\},
\]

and, for $\phi\in \mathfrak{F},$ $,$ and
\[
\mathcal{Z}(\phi) = \mathbb{E}\left[\phi(\mathbf{U} ; \mathbf{X})\right] \, , \quad \mathcal{Z}^*_n(\phi) = \frac{1}{n}\sum_{i=1}^n \phi(\mathbf{U}_i ; \mathbf{X}_i) \quad \text{and} \quad \widehat{\mathcal{Z}}_n(\phi) = \frac{1}{n}\sum_{i=1}^n \phi(\widehat{ \mathbf{U}}_i ; \mathbf{X}_i) \, .
\]

\subsubsection{Bracketing numbers}
\label{sec:bracket}

We first introduce the concept of bracketing numbers to measure the complexity of a class of functions $\mathcal{F}$. For $\varepsilon>0$, a $\varepsilon-$bracket $[\mathfrak{a},\mathfrak{b}]$ is the set of functions $f$ such that for all $ \mathbf{x} \in \mathbb{R}^d$,  $\mathbf{u} \in \mathbb{R}^k$, $\mathfrak{a}(\mathbf{u},\mathbf{x})\leq f(\mathbf{u},\mathbf{x})\leq \mathfrak{b}(\mathbf{u},\mathbf{x}),$ with the condition that
\[
\int (\mathfrak{a}(\mathbf{u},\mathbf{x})-\mathfrak{b}(\mathbf{u},\mathbf{x}))^2\mathrm{d}\mathbb{P}(\mathbf{u},\mathbf{x})\leq \varepsilon^2  \, .
\]
We then define $\mathcal{N}(\varepsilon,\mathcal{F})$ as the minimal number of $\varepsilon$-brackets required to cover the class of functions $\mathcal{F}.$ More details on bracketing numbers can be found in \citep[][Chapter 19]{vandervaart}, and \cite[][Chapter 2.2]{vandervaartwellner}.

\begin{lemma}
\label{lemma_GC}
For $\varepsilon>0$,
\[
\mathcal{N}(\varepsilon,\mathfrak{F})\leq \left(\frac{K^{m/2}C_1\|\Phi\|_2^m}{\varepsilon^m}\right)^K,
\]
for some constant $C_1$ depending only on $\Theta$ and $m$.
\end{lemma}

\begin{proof}

Consider an element $\phi\in \mathfrak{F}.$ It can be written as
$\phi=\sum_{\ell=1}^K \varphi_\ell\mathbf{1}_{\mathbf{x}\in \mathcal{T}_\ell},$
where each $\varphi_\ell$ is in $\mathcal{F},$ and satisfies Condition~\ref{cond:1}. Then, from \citep[][Example 19.7]{vandervaart}, for each $\ell=1,\ldots, K$, for all $\varepsilon >0$
\[
\mathcal{N}(\varepsilon,\mathcal{F})\leq \frac{C_1(m,\Theta)\|\Phi\|_2^m}{\varepsilon^m},
\]
where $C_1$ is a constant depending on $\mathrm{diam}(\Theta)$ and $m$.

Therefore, for each $\varphi_\ell$, the set of  $\varepsilon K^{-1/2}$-brackets $[a_{i(\ell)},b_{i(\ell)}]$ for $i=1, \ldots, K^{m/2}C_1(m,\Theta)\|\Phi\|_2^m\varepsilon^{-m}$  with $a_{i(\ell)}\leq b_{i(\ell)}$ covers $\mathcal F$. 

Now, for $i=1, \ldots, K^{m/2}C_1(m,\Theta)\|\Phi\|_2^m\varepsilon^{-m}$, define
\begin{equation}
\mathfrak{a}_i(\mathbf{u}, \mathbf{x})=\sum_{\ell=1}^K a_{i(\ell)}(\mathbf{u})\mathbf{1}_{\mathbf{x}\in \mathcal{T}_\ell} \quad ; \quad \mathfrak{b}_i(\mathbf{u}, \mathbf{x})=\sum_{\ell=1}^K b_{i(\ell)}(\mathbf{u})\mathbf{1}_{\mathbf{x}\in \mathcal{T}_\ell}. \label{new_brackets}
\end{equation}
Clearly,  $\mathfrak{a}_i\leq \mathfrak{b}_i$ and $\phi\in [\mathfrak{a}_i,\mathfrak{b}_i].$ Moreover,
\[
\int (\mathfrak{b}_i(\mathbf{u},\mathbf{x})-\mathfrak{a}_i(\mathbf{u},\mathbf{x}))^2 \mathrm{d}\mathbb{P}(\mathbf{u},\mathbf{x})\leq \sum_{\ell=1}^K \int (b_{i(\ell)}(\mathbf{u})-a_{i(\ell)}(\mathbf{u}))^2  \mathrm{d}\mathbb{P}(\mathbf{u})\leq \varepsilon^2.
\]

Thus, the set of brackets $[\mathfrak{a}_i,\mathfrak{b}_i]$, for $i=1, \ldots, K^{m/2}C_1(m,\Theta)\|\Phi\|_2^m\varepsilon^{-m}$ defined in \eqref{new_brackets} and deduced from the brackets $[a_{i(\ell)},b_{i(\ell)}]$ are $\varepsilon$-brackets covering $\mathfrak{F},$ and their number is less that
\[
\left(\frac{K^{m/2}C_1(m,\Theta)\|\Phi\|_2^m}{\varepsilon^m}\right)^K,
\]
leading to the result.

\end{proof}

\subsubsection{General results on sums involving pseudo-observations}
\label{sec:proof_pseudo}

The first result of this section shows how to replace pseudo-observations  $\mathbf{U}_i$ by their estimated version $\widehat{\mathbf{U}}_i$ in studying the asymptotic behavior of sums involving these quantities. Going back to $\mathbf{U}_i$ then simplifies considerably the study of such quantities, since one goes back to classical i.i.d. quantities.

\begin{lemma}
\label{lemma_cpobs}
Assume furthermore that  there exist $0\leq \beta_1,\beta_3< 1/2,$,  $0\leq 1 \beta_2 <1$ and two universal constants $A_1$ and $A_2$ such that for all $\varphi\in \mathcal{F}$ satisfies Condition~\ref{cond:2}. 

Then, under Assumption~\ref{a:pseudo}, 
\[
\sup_{\phi\in \mathfrak{F}} \left|\widehat{\mathcal Z}_n(\phi) - \mathcal Z^*_n(\phi) \right| 
=O_P(\varepsilon_n),
\]
with $\varepsilon_n$ tends to 0 when $n$ tends to $\infty$.
\end{lemma}

\begin{proof}
First, recall that 
\[
\sup_{\phi\in \mathfrak{F}} \left|\widehat{\mathcal Z}_n(\phi) - \mathcal Z^*_n(\phi) \right| = \sup_{\phi\in \mathfrak{F}}\left|\frac{1}{n}\sum_{i=1}^n \left\{\phi(\widehat{\mathbf{U}}_i;\mathbf{X}_i)-\phi(\mathbf{U}_i;\mathbf{X}_i)\right\}\right| \, .
\]
Then, from a Taylor expansion,
\[
\frac{1}{n}\sum_{i=1}^n \left\{\phi(\mathbf{\widehat{U}}_i;\mathbf{X}_i)-\phi(\mathbf{U}_i;\mathbf{X}_i)\right\}=\frac{1}{n}\sum_{i=1}^n \sum_{j=1}^k \partial_j \phi(U_{i}^{(1)},\ldots, \widetilde{U}_i^{(j)},\ldots,U_i^{(k)};\mathbf{X}_i)\left[\widehat{U}_i^{(j)}-U_i^{(j)}\right],
\]
where $\widetilde{U}_i^{(j)}$ is between $U_i^{(j)}$ and $\widehat{U}_i^{(j)}.$
From Condition~\ref{cond:2},
\[
|\partial_j\varphi(\mathbf{u})| \leq  \frac{A_2}{[u^{(j)}(1-u^{(j)})]^{\beta_2}} \sum_{r=1}^k \frac{1}{[u^{(r)}(1-u^{(r)})]^{\beta_3}}, \quad  j\in \{1,\ldots,k\},
\]
Thus, 
\[
|\partial_j \phi(U_{i}^{(1)},\ldots, \widetilde{U}_i^{(j)},\ldots,U_i^{(k)};\mathbf{X}_i)| \leq \frac{A_2}{[\tilde{U}_i^{(j)}(1-\tilde{U}_i^{(j)})]^{\beta_2}}\times \left\{\sum_{r\neq j}\frac{1}{[{U}_i^{(r)}(1-{U}_i^{(r)})]^{{\beta_3}}}+\frac{1}{[\tilde{U}_i^{(j)}(1-\tilde{U}_i^{(j)})]^{{\beta_3}}}\right\}.
\]
Note that
\[
\frac{1}{\widetilde{U}_i^{(j)}(1-\widetilde{U}_i^{(j)})}\leq \sup_{i=1,\ldots,n} \left(\frac{U^{(j)}_i}{\widehat{U}^{(j)}_i} +\frac{1-U^{(j)}_i}{1-\widehat{U}^{(j)}_i} \right) \frac{1}{U_i^{(j)}(1-U_i^{(j)})},
\]
leading to
\[
|\partial_j \phi(U_{i}^{(1)},\ldots, \widetilde{U}_i^{(j)},\ldots,U_i^{(k)};\mathbf{X}_i)| \leq 
\frac{ A_2[\max(1,\left(\frac{U^{(j)}_i}{\widehat{U}^{(j)}_i} +\frac{1-U^{(j)}_i}{1-\widehat{U}^{(j)}_i} \right) ]^{\beta_3}}{[{U}_i^{(j)}(1-{U}_i^{(j)})]^{\beta_2}}\left\{\sum_{r=1}^k\frac{1}{[{U}_i^{(r)}(1-{U}_i^{(r)})]^{{\beta_3}}}\right\}.
\]
Let 
\[
Z_i=\sum_{r=1}^k\frac{1}{[U_i^{(r)}(1-U_i^{(r)})]^{\beta_3}}.
\]
Since $\beta_3<1/2,$ $E[Z_i^2]<\infty.$
Hence,
\begin{eqnarray*}
\left|\widehat{\mathcal Z}_n(\phi) - \mathcal Z^*_n(\phi)\right| &\leq & \frac{A_2}{n} \sum_{j=1}^k \sum_{i=1}^n \frac{Z_i}{[{U}_i^{(j)}(1-{U}_i^{(j)})]^{\beta_2-\beta'}} \sup_{\substack{i=1,\ldots,n\\{j=1,\ldots,k}}}\left|\frac{\widehat{U}^{(j)}_i-U^{(j)}_i}{\left[U_i^{(j)}(1-U_i^{(j)})\right]^{\beta'}}\right| \\
&& \times \sup_{\substack{i=1,\ldots,n\\{j=1,\ldots,k}}}\left(\frac{U^{(j)}_i}{\widehat{U}_i^{(j)}}+\frac{1-U^{(j)}_i}{1-\widehat{U}^{(j)}_i}\right)^{\beta_3}
\end{eqnarray*}
with $\beta'=\min(\beta_3,\alpha)$. Then, first, from Cauchy-Schwarz inequality,
\[
\mathbb{E}\left[\left\{\frac{Z_i}{[{U}_i^{(j)}(1-{U}_i^{(j)})]^{\beta_2-\beta'}}\right\}^2\right]\leq \mathbb{E}[Z_i^2]^{1/2}\mathbb{E}\left[\frac{1}{[{U}_i^{(j)}(1-{U}_i^{(j)})]^{2[\beta_2-\beta']}}\right]^{1/2} <\infty \,.
\]
Second, from Assumption~\ref{a:pseudo}, 
\[
\sup_{\substack{i=1,\ldots,n\\{j=1,\ldots,k}}}\left|\frac{\widehat{U}^{(j)}_i-U^{(j)}_i}{\left[U_i^{(j)}(1-U_i^{(j)})\right]^{\beta'}}\right| = O_p(\varepsilon_n) \, ,
\]
and
\[
\sup_{\substack{i=1,\ldots,n\\{j=1,\ldots,k}}}\left|\frac{U^{(j)}_i}{\widehat{U}_i^{(j)}}+\frac{1-U^{(j)}_i}{1-\widehat{U}^{(j)}_i}\right|^{\beta_3} = O_P(1) \, .
\]


\end{proof}

\begin{remark}
\label{rem_trimming} If (\ref{a11}) does not hold, the estimation procedure can be modified by introducing some trimming, that is multiplying each term of the log-likelihood by $\mathbf{1}_{\min(1-\hat{U}_i^{(j)},\hat{U}_i^{(j)})\geq \eta_n}.$ If $\eta_n$ tends to zero slower than $\varepsilon_n,$ $\hat{U}_i^{(j)}
\geq U_i^{(j)}/2$ for $n$ large enough due to (\ref{a12}) for the indexes $i$ where this indicator function is not zero. However, the introduction of trimming induces some bias for the estimator, which can be controlled thanks to Assumption \ref{assum:llk}.
\end{remark}

With at hand Lemma~\ref{lemma_cpobs} and the complexity bound of Lemma~\ref{lemma_GC}, one can derive the main result of this section, that will be used several times in the proof of our main theorems.

\begin{proposition}
\label{p1}
Assume furthermore that  there exist $0\leq \beta_1,\beta_3< 1/2$,  $0\leq / \beta_2 <1$ and two universal constants $A_1$ and $A_2$ such that for all $\varphi\in \mathcal{F}$ satisfies Condition~\ref{cond:2}. Then,
\[
\sup_{\phi\in \mathfrak{F}}\left|\widehat{\mathcal{Z}}_n(\phi)-\mathcal{Z}(\phi)\right|=O_P\left(\sqrt{\frac{K\log K}{n}}+\varepsilon_n\right).
\]
\end{proposition}

\begin{proof}
Writing
\[
\sup_{\phi\in \mathfrak{F}}\left|\widehat{\mathcal{Z}}_n(\phi)-\mathcal{Z}(\phi)\right| \leq \sup_{\phi\in \mathfrak{F}} \left|\widehat{\mathcal{Z}}_n(\phi)-\mathcal{Z}_n^*(\phi)\right| + \sup_{\phi\in \mathfrak{F}} \left|\mathcal{Z}_n^*(\phi) - \mathcal Z(\phi)\right|
\]
For the first term, from Lemma~\ref{lemma_cpobs}, 
\[
\sup_{\phi\in \mathfrak{F}}\left|\widehat{\mathcal{Z}}_n(\phi)-\mathcal{Z}_n^*(\phi)\right|=O_P\left(\varepsilon_n\right).
\]
For the second term, introduce for $\delta>0$, 
$J(\delta,\mathfrak{F})= \int_0^{\delta} \sqrt{\log \mathcal{N}(\varepsilon,\mathfrak{F})}\mathrm{d}\varepsilon.$
From \citep[Corollary 19.35][]{vandervaart},
\[
\sqrt{n}\mathbb{E}\left[\sup_{\phi\in \mathfrak{F}}\left|\mathcal{Z}^*_n(\phi)-\mathcal{Z}(\phi)\right|\right]\leq A_3 J(\|\Phi\|_2 ,\mathfrak{F}),
\]
for some universal constant $A_3\geq 0.$ Then, from Lemma~\ref{lemma_GC},
\[
J(\|\Phi\|_2,\mathfrak{F})\leq \int_0^{\|\Phi\|_2} K^{1/2}\left\{\frac{m}{2}\log K+\log (C_1\|\Phi\|_2^m)+m \log \left(\frac{1}{\varepsilon}\right)\right\}^{1/2}d\varepsilon.
\]
Hence,
\[
\sqrt{n}\mathbb{E}\left[\sup_{\phi\in \mathfrak{F}}\left|\mathcal{Z}^*_n(\phi)-\mathcal{Z}(\phi)\right|\right]\leq C_2(m,\Theta)\sqrt{K\log K}.
\]

\end{proof}

\subsection{Proof of Proposition~\ref{prop_consistency}}

We are now ready to prove Proposition~\ref{prop_consistency}. 
 \label{result_consistency}

Recall that
\[
\|\widehat{\bm\theta}(\cdot|\mathbb{T})-\bm\theta^0(\cdot|\mathbb{T})\|_1=\sum_{\ell=1}^K |\widehat{\bm\theta}_\ell-\bm\theta^*_\ell| \mathbb{P}(\mathbf{X}\in \mathcal{T}_\ell),
\]
so that it suffices to show that
\begin{equation}
\label{eq_cons}
\sup_{\ell=1,\ldots,K}|\widehat{\bm\theta}_\ell-\bm\theta^0_\ell|=o_P(1).
\end{equation}

Let 
\begin{eqnarray*}
\widehat{\mathcal{L}}_n (\theta_1,\ldots,\theta_K) &=& \frac{1}{n}\sum_{i=1}^n \sum_{\ell=1}^K\log c_{\theta_\ell}(\widehat{\mathbf{U}}_i)\mathbf{1}_{\mathbf{X}_i\in \mathcal{T}_\ell}, \\
\mathcal{L}^*_n (\theta_1,\ldots,\theta_K) &=& \frac{1}{n}\sum_{i=1}^n \sum_{\ell=1}^K\log c_{\theta_\ell}(\mathbf{U}_i)\mathbf{1}_{\mathbf{X}_i\in \mathcal{T}_\ell}, \\
\mathcal{L} (\theta_1,\ldots,\theta_K) &=& \mathbb{E}\left[ \sum_{\ell=1}^K \log c_{\theta_\ell}(\mathbf{U}_i)\mathbf{1}_{\mathbf{X}_i\in \mathcal{T}_\ell}\right].
\end{eqnarray*}

From \citep[][Corollary 3.2.3]{vandervaartwellner}, \eqref{eq_cons} holds if
\[
\sup_{\theta_1,\ldots,\theta_\ell} |\widehat{\mathcal{L}}_n (\theta_1,\ldots,\theta_\ell) - \mathcal{L} (\theta_1,\ldots,\theta_\ell)|=o_P(1).
\]
Let us introduce
\begin{equation}
\label{class_F}
\mathfrak{F}_1=\left\{(\mathbf{u},\mathbf{x})\mapsto \sum_{\ell=1}^K\log c_{\bm{\theta}}(\mathbf{u})\mathbf{1}_{\mathbf{x}\in \mathcal{T}_\ell} \text{ with } \bm{\theta}=(\theta_\ell)_{\ell=1,\ldots,K}\in \Theta^K \right\}.
\end{equation}
From Assumption~\ref{assum:llk}, Proposition~\ref{p1} applies to $\mathfrak{F}_1,$ leading to 
\[
\sup_{\theta_1,\ldots,\theta_\ell} |\widehat{\mathcal{L}}_n (\theta_1,\ldots,\theta_\ell) - \mathcal{L} (\theta_1,\ldots,\theta_\ell)|=\sup_{\phi\in \mathfrak{F}_1}\left|\mathcal{Z}_n^*(\phi)-\mathcal{Z}(\phi)\right| = O_P\left(\sqrt{\frac{K\log K}{n}}+\varepsilon_n\right)
\] 
which tends to zero under the condition on $K$ and the result follows.

\subsection{Proof of Theorem~\ref{th_main}}
\label{proof_th_main}

Introduce 
\begin{eqnarray*}
\dot{\mathcal{L}}_{n}(\theta_1,\ldots,\theta_K) &=& \frac{1}{n}\sum_{\ell=1}^{K}\sum_{i=1}^n \nabla_{\bm\theta}\log c_{\theta_\ell}(\widehat{\mathbf{U}}_i)\mathbf{1}_{\mathbf{X}_i\in \mathcal{T}_\ell}, \\
\dot{\mathcal{L}}(\theta_1,\ldots,\theta_K) &=& \sum_{\ell=1}^K  \mathbb{E}\left[\nabla_{\bm\theta}\log c_{\theta_\ell}(\mathbf{U})\mathbf{1}_{\mathbf{X}\in \mathcal{T}_\ell}\right],
\end{eqnarray*}
and
\[
\mathfrak{F}_2=\left\{(\mathbf{u},\mathbf{x})\rightarrow \sum_{\ell=1}^K \nabla_{\bm\theta}\log c_{\theta_\ell}(\mathbf{u})\mathbf{1}_{\mathbf{x}\in \mathcal{T}_\ell}:(\theta_\ell)_{l=1,\ldots,K}\in \Theta^K\right\}.
\]

From Proposition~\ref{p1},
\begin{equation}\label{eq:rate}\sup_{\theta_1,\ldots,\theta_\ell} \left|\dot{\mathcal{L}}_{n}(\theta_1,\ldots,\theta_K)-\dot{\mathcal{L}}(\theta_1,\ldots,\theta_K)\right|=O_P\left(\frac{[K\log K]^{1/2}}{n^{1/2}}+\varepsilon_n\right).
\end{equation}
Then, write
\begin{eqnarray*}
\dot{\mathcal{L}}(\theta^0_1,\ldots,\theta^0_K)-\dot{\mathcal{L}}(\widehat{\theta}_1,\ldots,\widehat{\theta}_K)&=&\left\{\dot{\mathcal{L}}(\theta^0_1,\ldots,\theta^0_K)-\dot{\mathcal{L}_n}(\theta^0_1,\ldots,\theta^0_K)\right\}\\
&&+\left\{\dot{\mathcal{L}}_n(\theta^0_1,\ldots,\theta^0_K)-\dot{\mathcal{L}_n}(\widehat{\theta}_1,\ldots,\widehat{\theta}_K)\right\} \\
&& +\left\{\dot{\mathcal{L}}_n(\widehat{\theta}_1,\ldots,\widehat{\theta}_K)-\dot{\mathcal{L}}(\widehat{\theta}_1,\ldots,\widehat{\theta}_K)\right\}.
\end{eqnarray*}
The rate of the first and last brackets in this decomposition are given by \eqref{eq:rate}, while the middle one is
\[
\left\{\dot{\mathcal{L}}_n(\theta^0_1,\ldots,\theta^0_K)-\dot{\mathcal{L}_n}(\widehat{\theta}_1,\ldots,\widehat{\theta}_K)\right\}=\dot{\mathcal{L}}_n(\theta^0_1,\ldots,\theta^0_K),
\]
has also the same rate. This shows that
\[
\left|\dot{\mathcal{L}}(\theta^0_1,\ldots,\theta^0_K)-\dot{\mathcal{L}}(\widehat{\theta}_1,\ldots,\widehat{\theta}_K)\right|= O_P\left(\frac{[K\log K]^{1/2}}{n^{1/2}}+\varepsilon_n\right).
\]

From the assumption on the Hessian matrix, and since Proposition~\ref{prop_consistency} applies (which guarantees that each $\widehat{\theta}_\ell$ is in an arbitrary small neighborhood of $\theta^*_\ell$ for $n$ large enough), we get
\[
\left|\dot{\mathcal{L}}(\theta^0_1,\ldots,\theta^0_K)-\dot{\mathcal{L}}(\widehat{\theta}_1,\ldots,\widehat{\theta}_K)\right|\geq \mathfrak{a}\|\widehat{\theta}-\theta^0\|_1,
\]
for some $\mathfrak{a}>0,$ from a Taylor expansion, and the result follows.

\subsection{Proof of Theorem~\ref{th_main2}}
\label{proof_oracle}

Let $\widehat{\bm\theta}^K$ denote the best tree with $K$ leaves, with respect to the log-likelihood (and $\bm\theta^{0,K}$ its corresponding limit), and $\widehat{K}$ denote the number of leaves of $\bar{\bm\theta}.$

Write
\[
\bar{\bm\theta}-\bm\theta^0=[\widehat{\bm\theta}^{K^0}-\bm\theta^0]\mathbf{1}_{\widehat{K}=K^0}+\sum_{K\neq K^0} \left[\widehat{\bm\theta}^K-\bm\theta^{0}\right]\mathbf{1}_{\widehat{K}=K}.
\]
Let $R=\sum_{K\neq K^0} \left[\widehat{\bm\theta}^K-\bm\theta^{0,K}\right]\mathbf{1}_{\widehat{K}=K},$ and note that
\[
\mathbb{P}(R\geq t)\leq \mathbb{P}(\widehat{K}>K^0)+\mathbb{P}(\widehat{K}<K^0).
\]

The result is then shown if we prove that $\mathbb{P}(\widehat{K}>K^0)$ and $\mathbb{P}(\widehat{K}<K^0)$ tend to zero when $n$ tends to infinity, which is done below studying each probability separately.

\paragraph{First case: $\mathbb{P}(\widehat{K}>K^0).$}

We will use the notation $\mathcal{L}_n^K$ to denote the log-likelihood associated with $\widehat{\bm\theta}^K.$ If $\widehat{K}> K^0,$ this means that there exists some $K^0<K<K_{\max}$ such that
\[
\mathcal{L}_n^K-\mathcal{L}_n^{K^0}\geq \lambda (K-K^0),
\]
that is
\[
\mathcal{L}_n^K(\widehat{\bm\theta}^K)-\mathcal{L}_n^{K}(\bm\theta^{0,K})\geq \lambda (K-K^0),
\]
since $\mathcal{L}_n^K(\bm\theta^0)=\mathcal{L}_n^{K^0}(\bm\theta^{0,K})$ for $K\geq K^0$. Whence, 
\[
\mathbb{P}(\widehat{K}>K^0)\leq \mathbb{P}(\exists K>K^0: \mathcal{L}_n^K(\widehat{\bm\theta}^K)-\mathcal{L}_n^{K}(\bm\theta^{0,K})\geq \lambda (K-K^0)).
\]
Since $\lambda (K-K^0)\geq \lambda,$ and since $\mathcal{L}_n^K(\widehat{\bm\theta}^K)-\mathcal{L}_n^{K}(\bm\theta^{0,K})\leq \mathcal{L}_n^{K_{\max}}(\widehat{\bm\theta}^{K_{\max}})-\mathcal{L}_n^{K_{\max}}(\bm\theta^{0,K_{\max}}),$
\begin{equation}
\mathbb{P}(\widehat{K}>K^0)\leq \mathbb{P}\left(\mathcal{L}_n^{K_{\max}}(\widehat{\bm\theta}^{K_{\max}})-\mathcal{L}_n^{K_{\max}}(\bm\theta^{0,K_{\max}})\geq \lambda \right). \label{eureka1}
\end{equation}
In the proof of Proposition~\ref{prop_consistency}, we showed that 
\[
\mathcal{L}_n^{K_{\max}}(\widehat{\bm\theta}^{K_{\max}})-\mathcal{L}_n^{{K}_{\max}}(\bm\theta^{0,K})=O_P([K_{\max}\log K_{\max}]^{1/2}n^{-1/2} + \varepsilon_n).
\]
Hence, the right-hand side of \eqref{eureka1} tends to zero provided that $\lambda n^{1/2}[K_{\max}\log K_{\max}]^{-1/2}\rightarrow \infty.$
\\\\
\textbf{Second case: $\mathbb{P}(\widehat{K}<K^0).$}
\\
In this case $\mathcal{L}_n^K-\mathcal{L}_n^{K^0}\leq \mathcal{L}_n^{(K^0-1)}-\mathcal{L}_n^{K^0}.$ 
From the proof of Proposition~\ref{prop_consistency}, $\mathcal{L}_n^{(K^0-1)}-\mathcal{L}^{*(K^0-1)}=O_P([K^{*}\log K^0]^{1/2}n^{-1/2}),$ and $\mathcal{L}_n^{(K^0)}-\mathcal{L}^{*(K^0)}=O_P([K^{*}\log K^0]^{1/2}n^{-1/2}).$ Then, similarly to the first case,
\[
\mathbb{P}(\widehat{K}<K^0)\leq \mathbb{P}\left(\mathcal{L}_n^{(K^0-1)}-\mathcal{L}^{*(K^0-1)}-\mathcal{L}_n^{K^0}+\mathcal{L}^{*K^0}\geq \frac{\lambda}{2}\right)+\mathbb{P}\left(\mathcal{L}^{*(K^0-1)}-\mathcal{L}^{*K^0}\geq \frac{\lambda}{2}\right).
\]
The first probability tends to zero under the same conditions as in the first case, while the second is equal to $\mathbf{1}_{\mathcal{L}^{*(K^0-1)}-\mathcal{L}^{*(K^0)}\geq \lambda/2},$ since the quantity $\mathcal{L}^{*(K^0-1)}-\mathcal{L}^{*K^0}$ is deterministic. This indicator function tends to zero when $n$ tends to infinity if $\lambda$ tends to zero.

\subsection{Convergence rate for the margins for kernel estimators}
\label{sec:margins1}

In this section, we show that Assumption~\ref{a:pseudo} holds for the kernel estimator \eqref{kernelF}.
This is in fact a consequence of Theorem 4 in \citep{einmahl2005uniform}. We show the result under three additional assumptions on the model:
\begin{enumerate}
    \item the density of $\mathbf{X}$ is bounded away from zero on $\mathcal{X},$ that is $\inf_{\mathbf{x}\in \mathcal{X}}f_{\mathbf{X}}(\mathbf{x})>0;$
    \item we have 
    \[
    \sup_{\mathbf{x\in \mathcal{X}},y}\left|\frac{F^{(j)}(y)}{F^{(j)}(y|\mathbf{x})}+\frac{1-F^{(j)}(y)}{1-F^{(j)}(y|\mathbf{x})}\right|\leq \mathfrak{a},
    \]
    for some finite constant $\mathfrak{A};$
    \item the kernel function is a continuous and bounded function, symmetric around 0, such that $\int u^2 K(u)du<\infty,$ the density $\mathbf{x}\mapsto f_{\mathbf{X}}(\mathbf{x})$ and $\mathbf{x}\mapsto F^{(j)}(t|\mathbf{x})$ are twice continuously differentiable with respect to $\mathbf{x}$, with uniformly bounded derivatives up to order 2.
\end{enumerate}
The first assumption is required to avoid the denominator, in the kernel weights, going too close to zero. The second one is a way to consider that there is some kind of uniform domination of the behavior of the conditional distributions when $\mathbf{x}$ changes. Finally, the third assumption is classical in kernel regression and will help to control the bias term involved in smoothing techniques.

Introducing the kernel estimator of the density of $\mathbf{X},$ 
\[
\widehat{f}_X(\mathbf{x})=\frac{1}{nh^d}\sum_{i=1}^n K\left(\frac{\mathbf{X}_i-\mathbf{x}}{h}\right),
\]
we can write, for $t\leq 1/2,$
\[
\widehat{f}_X(\mathbf{x})\frac{\widehat{F}^{(j)}(t|\mathbf{x})}{\left[F^{(j)}(t|\mathbf{x})(1-F^{(j)}(t|\mathbf{x}))\right]^{\alpha}}=\frac{1}{nh^p}\sum_{i=1}^n K\left(\frac{\mathbf{X}_i-\mathbf{x}}{h}\right)f_t(Y_i^{(j)}),
\]
where
\[
f_t(y)=\frac{\mathbf{1}_{y\leq t}}{\left[F^{(j)}(t|\mathbf{x})(1-F^{(j)}(t|\mathbf{x}))\right]^{\alpha}}\leq \frac{1}{[F^{(j)}(y|\mathbf{x})]^{\alpha}[1-F^{(j)}(1/2|\mathbf{x})]^{\alpha}}\leq \frac{\mathfrak{A}^{\alpha}}{[F^{(j)}(y)]^{\alpha}[1-F^{(j)}(1/2|\mathbf{x})]^{\alpha}}.
\]
Since 
\[
\mathbb{E}\left[\left(\frac{1}{[F^{(j)}(Y_i^{(j)})]^{\alpha}}\right)^p\right]<\infty,
\]
for some $p>2$ for $\alpha<1/2,$ and since the covering number of the class of functions $f_t$ is controlled  \citep[see][Example 19.12]{vandervaart}),then Theorem 4 of  \citep{einmahl2005uniform} applies, showing that
\[
\sup_{t\leq 1/2,x}\left|\frac{1}{nh^d}\sum_{i=1}^n K\left(\frac{\mathbf{X}_i-\mathbf{x}}{h}\right)f_t(Y_i^{(j)})-\mathbb{E}\left[f_t(Y_i^{(j)})K\left(\frac{\mathbf{X}_i-\mathbf{x}}{h}\right)\right]\right|=O_P\left([\log n]^{1/2}n^{-1/2}h^{-d/2}\right).
\]
Then, from a Taylor expansion and the third assumption of this section, we get
\[
\mathbb{E}\left[f_t(Y_i^{(j)})K\left(\frac{\mathbf{X}_i-\mathbf{x}}{h}\right)\right]=\mathbb{E}\left[f_t(Y_i^{(j)})|\mathbf{X}_i=\mathbf{x}\right]f_{\mathbf{X}}(\mathbf{x})+O(h^2).
\]
Let us note that this $h^2$ rate can be improved if one uses a degenerate kernel with a sufficiently high number of moments equal to zero. Then, from the rate of uniform convergence of $\widehat{f}_X(\mathbf{x})$ \citep[from Theorem 1][]{einmahl2005uniform}), we get
\[
\sup_{t\leq 1/2,\mathbf{x}}\left|\frac{\widehat{F}^{(j)}(t|\mathbf{x})-F^{(j)}(t|\mathbf{x})}{\left[F^{(j)}(t|\mathbf{x})(1-F^{(j)}(t|\mathbf{x}))\right]^{\alpha}}\right|=O_P(h^2+[\log n]^{1/2}n^{-1/2}h^{-d/2}).
\]
Studying the supremum for $t>1/2$ can be done in the same way, by studying $1-F^{(j)}$ instead of $F^{(j)}.$

\subsection{Convergence rate for the margins for discrete covariates}
\label{sec:margins2}

For discrete covariates, recall that
\[
\widehat{F}^{(j)}(t|\mathbf{x})=\frac{\sum_{i=1}^n \mathbf{1}_{Y_i^{(j)}\leq t}\mathbf{1}_{\mathbf{X}_i\in \mathcal{C}(\mathbf{x})}}{\sum_{i=1}^n \mathbf{1}_{\mathbf{X}_i\in \mathcal{C}(\mathbf{x})}}.
\]
From the central limit theorem,
\[
\frac{1}{n}\sum_{i=1}^n \mathbf{1}_{\mathbf{X}_i\in \mathcal{C}(\mathbf{x})}=\mathbb{P}\left(\mathbf{X}\in \mathcal{C}(\mathbf{x})\right)+O_P(n^{-1/2}).
\]
The upper part can be studied using similar arguments as  \citep[][Example 19.12]{vandervaart}, noticing that the class of functions $f_t(Y_i^{(j)})\mathbf{1}_{\mathbf{X}_i\in \mathcal{C}(\mathbf{x})}$ (where $f_t$ is defined in section~\ref{sec:margins1}) has a similar covering number as the class of functions $f_t.$ This leads to
\[
\sup_{t,\mathbf{x}}\left|\frac{1}{n}\sum_{i=1}^n \mathbf{1}_{Y_i^{(j)}\leq t}\mathbf{1}_{\mathbf{X}_i\in \mathcal{C}(\mathbf{x})}-F^{(j)}(t|\mathbf{x})\mathbb{P}\left(\mathbf{X}\in \mathcal{C}(\mathbf{x})\right)\right|=O_P(n^{-1/2}).\]
Then, we get
\[
\sup_{t,\mathbf{x}}\left|\frac{\widehat{F}^{(j)}(t|\mathbf{x})-F^{(j)}(t|\mathbf{x})}{\left[F^{(j)}(t|\mathbf{x})(1-F^{(j)}(t|\mathbf{x}))\right]^{\alpha}}\right|=O_P(n^{-1/2}).
\]

\subsection{Regression trees for margin estimation in the real data example}

We report here the regression trees resulting from fitting the variables $(Y^{(1)}, Y^{(1)})$ as a function of the covariates year and Influenza Transmission Zone. (see Section~\ref{sec:flunet_data}).

\begin{figure}
\centering
\includegraphics[width=0.9\textwidth]{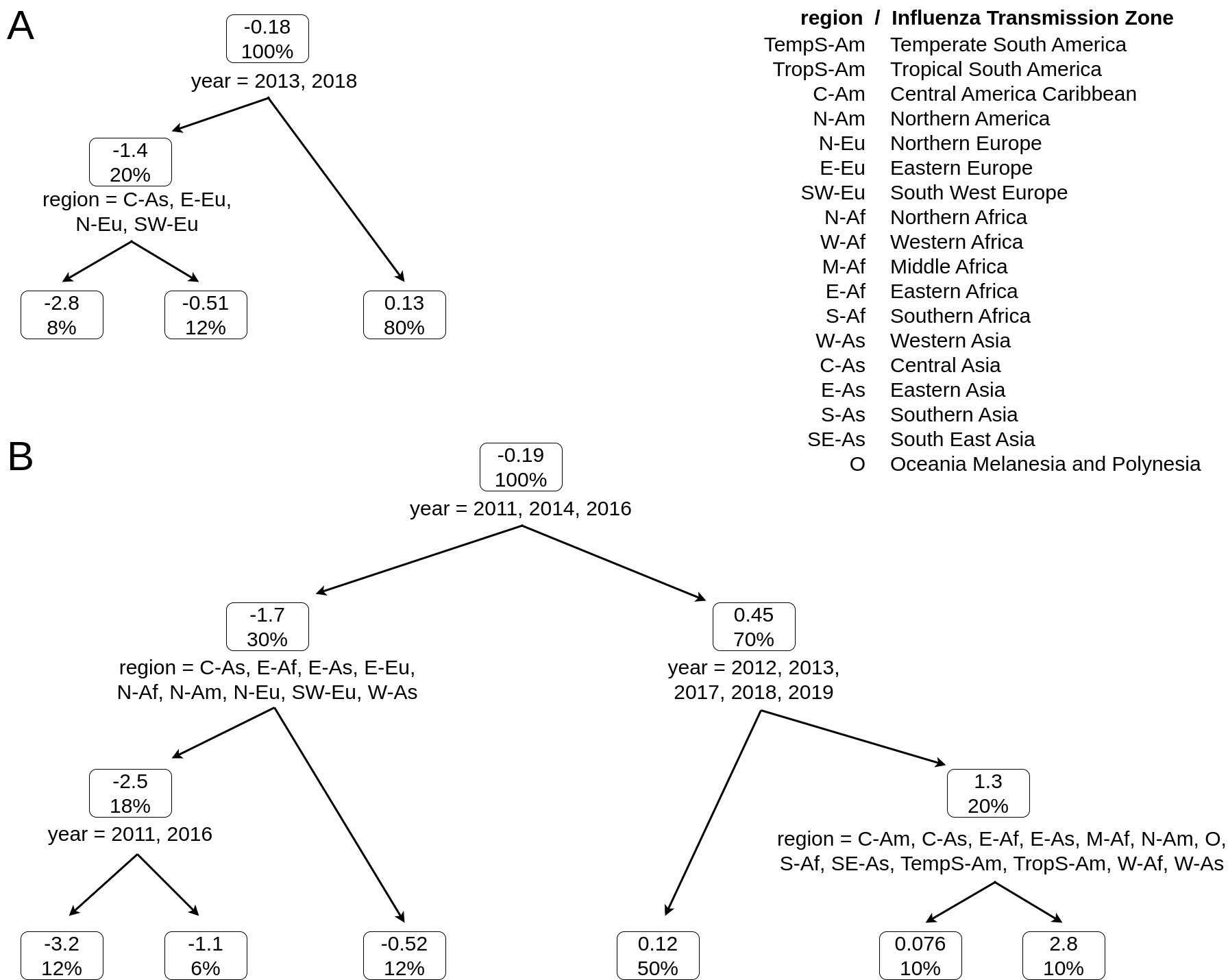}
\caption{Optimal trees for margins estimation. Country and Influenza Transmission Zones are classified by the regression trees to approximate the response variables $Y^{(1)}$ (plot A) and $Y^{(2)}$ (plot B). The coefficients of determination of the two fits are 0.29 and 0.5, respectively. For each node, the average value of the response variable and the percentage of the observations included are indicated. In the top-right corner, a legend illustrates the abbreviations used for the Influenza Transmission Zones.}
\label{fig:tree_margins}
\end{figure}


\bibliographystyle{abbrvnat} 
\bibliography{bibli}

\end{document}